\documentclass[12pt]{article}

\usepackage{amsmath}
\usepackage{amsfonts}
\usepackage{color}
\usepackage{float}   
\usepackage{amssymb}
\usepackage{bm}    
\usepackage{graphicx}   
\usepackage{enumerate}
\usepackage{amsthm,amscd}
\usepackage[all]{xy}
\usepackage{multirow,makecell} 
\usepackage{pifont}        
 \usepackage{arydshln}    
 \usepackage{booktabs}  

\makeatletter
\def\hlinew#1{%
  \noalign{\ifnum0=`}\fi\hrule \@height #1 \futurelet
   \reserved@a\@xhline}
\makeatother       

 \allowdisplaybreaks

\usepackage{hyperref}

\newtheorem{theorem}{Theorem}[section]

\newtheorem{conjecture}[theorem]{Conjecture}

\newtheorem{lemma}[theorem]{Lemma}
\newtheorem{claim}{Claim}[section] 
\newtheorem{proposition}[theorem]{Proposition}

\begin{document}

\title{The maximum spectral radius of non-bipartite graphs forbidding short odd cycles\thanks{The research was supported by National Natural Science Foundation of China grant 11931002. 
E-mail addresses: \url{ytli0921@hnu.edu.cn} (Y\v{o}ngt\={a}o L\v{i}), 
\url{ypeng1@hnu.edu.cn} (Yu\`{e}ji\`{a}n P\'{e}ng, corresponding author).}  }

\author{Yongtao Li, Yuejian Peng$^{\dag}$ \\[2ex]
{\small School of Mathematics, Hunan University} \\
{\small Changsha, Hunan, 410082, P.R. China }  }

\maketitle

\begin{abstract}
It is well-known that eigenvalues of graphs can be used to 
describe structural properties and parameters of graphs. 
A theorem of Nosal states that 
if $G$ is a triangle-free graph with $m$ edges, then 
$\lambda (G)\le \sqrt{m}$, equality holds if and only if 
$G$ is a complete bipartite graph. 
Recently, 
Lin, Ning and Wu [Combin. Probab. Comput. 30 (2021)] 
proved a generalization for non-bipartite triangle-free graphs. 
Moreover, Zhai and Shu [Discrete Math. 345 (2022)] presented  a further improvement. 
In this paper,  we present an alternative method for proving   
the improvement by Zhai and Shu. 
Furthermore, the method can allow us to give a  refinement on 
the result of Zhai and Shu for non-bipartite graphs without 
short odd cycles.  
 \end{abstract}

{{\bf Key words:}  Nosal theorem; 
eigenvalues; 
odd cycles; 
non-bipartite graphs;
Cauchy interlacing theorem. }

{{\bf 2010 Mathematics Subject Classification.}  05C50, 05C35.}

\section{Introduction}

The present work can be viewed as the second paper 
of our previous project \cite{LP2022second}. 
In this paper, 
we shall use the following standard notation; see e.g., the monograph \cite{BM2008}. 
We consider only simple and undirected graphs. Let $G$ be a simple
 graph with vertex set $V(G)=\{v_1, \ldots, v_n\}$ and edge set $E(G)=\{e_1, \ldots, e_m\}$.  
 We usually write $n$ and $m$ for the number of vertices and edges 
 respectively. 
Let $N(v)$ or $N_G(v)$ be the set of neighbors of $v$, 
and $d(v)$ or $d_G(v)$ be the degree of a vertex $v$ in $G$. 
For a subset $S\subseteq V(G)$, we write $e(S)$ for the 
number of edges with two endpoints in $S$. 
  Let $K_{s,t}$ be the complete bipartite graph with parts of sizes 
  $s$ and $t$. 
 We write  $C_n$ and $P_n$ for the cycle and 
 path on $n$ vertices respectively. 
We denote by $t(G)$ the number of triangles in $G$.

\subsection{The classical extremal graph problems}

We say that a graph $G$ is $F$-free if it does not contain 
  an isomorphic copy of $F$ as a subgraph. 
  Apparently, every bipartite graph is $C_{2k+1}$-free 
  for every integer $k\ge 1$. 
 The {\em Tur\'{a}n number} of a graph $F$ is the maximum number of edges  in an $n$-vertex $F$-free graph, and 
  it is usually  denoted by $\mathrm{ex}(n, F)$. 
  A graph on $n$ vertices with no subgraph $F$ and with $\mathrm{ex}(n, F)$ edges is called an {\em extremal graph} for $F$. 
As is known to all, the Mantel theorem \cite{Man1907} asserts that 
if $G$ is an $n$-vertex graph with  
at least $\lfloor \frac{n^2}{4} \rfloor$ edges, 
then either there exist three edges in $G$ that form a triangle 
or $G=K_{\lfloor \frac{n}{2}\rfloor, \lceil \frac{n}{2} \rceil }$, 
the balanced complete bipartite graph. 
  
  \begin{theorem}[Mantel, 1907] \label{thmMan}
  Let  $G$ be an $n$-vertex graph. If $G$ is triangle-free, then 
$ e(G) \le  
e(K_{\lfloor \frac{n}{2}\rfloor, \lceil \frac{n}{2} \rceil }) =  \lfloor \frac{n^2}{4} \rfloor $,  
equality holds if and only if  $G=K_{\lfloor \frac{n}{2}\rfloor, \lceil \frac{n}{2} \rceil }$.
\end{theorem}
  
Mantel's theorem  has many interesting applications and generalizations in the literature; see, e.g., 
\cite[pp. 269--273]{AZ2014} and \cite[pp. 294--301]{Bollobas78} for standard proofs, 
\cite{BT1981,Bon1983} for generalizations, and 
 \cite{FS13, Sim13} for recent comprehensive surveys. 
In particular, 
Mantel's Theorem \ref{thmMan} 
was refined in the sense of the following stability form. 

\begin{theorem}[Erd\H{o}s] \label{thmErd}
Let $G$ be an $n$-vertex triangle-free graph. 
If $G$ is not bipartite, then 
$ e(G)\le \lfloor \frac{(n-1)^2}{4} \rfloor+1 $. 
\end{theorem}

 \begin{figure}[H]
\centering 
\includegraphics[scale=0.8]{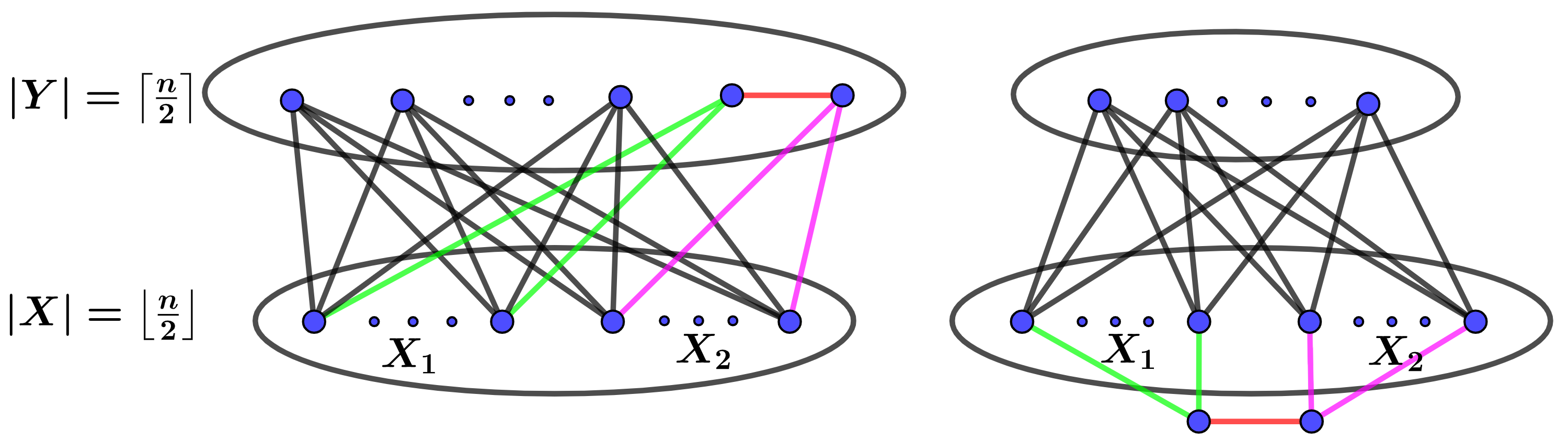}  
\caption{Extremal graphs in Theorem \ref{thmErd}.} \label{fig-1} 
\end{figure} 

It is said that this stability result attributes to Erd\H{o}s; 
see \cite[Page 306, Exercise 12.2.7]{BM2008}.  
The bound in Theorem \ref{thmErd} is best possible and 
the extremal graphs are not unique. 
To show that the bound is sharp for all integers $n$, 
we take two vertex sets 
$X$ and $Y$ with $|X|= \lfloor \frac{n}{2} \rfloor$ 
and $|Y|= \lceil \frac{n}{2} \rceil$. 
We take two  vertices $u,v \in Y$ and  join them, 
 then we put every edge between $X$ and $Y\setminus \{u,v\}$. 
We partition $X$ into two parts $X_1$ and $X_2$ {\it arbitrarily} 
(this shows that the extremal graph is not unique), 
then we connect $u$ to every vertex in $X_1$, and 
$v$ to every vertex in $X_2$;  see Figure \ref{fig-1}
This yields a graph $G$ 
which contains no triangle and 
 $e(G)= \lfloor \frac{n^2}{4}\rfloor - \lfloor \frac{n}{2}\rfloor +1
  = \lfloor \frac{(n-1)^2}{4} \rfloor +1 $. 
  Note that $G$ has a cycle $C_5$, so $G$ is not bipartite.

\subsection{The spectral extremal graph problems} 

There are various matrices that are associated with a graph, 
such as the adjacency matrix, the incidence matrix, the distance matrix, 
 the Laplacian matrix and signless Laplacian matrix. 
One of the main problems of algebraic graph theory 
is to determine the combinatorial properties of a graph that are reflected from 
the  algebraic properties of  its associated matrices. 
Let $G$ be a simple graph on $n$ vertices. 
The \emph{adjacency matrix} of $G$ is defined as 
$A(G)=[a_{ij}]_{n \times n}$ where $a_{ij}=1$ if two vertices $v_i$ and $v_j$ are adjacent in $G$, and $a_{ij}=0$ otherwise.   
We say that $G$ has eigenvalues $\lambda_1 , \lambda_2,\ldots ,\lambda_n$ if these values are eigenvalues of 
the adjacency matrix $A(G)$. 
We denote by $\lambda_i(G)$  the {\it $i$-th largest eigenvalue} of $G$. 
Let $\lambda (G)$ be the maximum  value in absolute 
 among all eigenvalues of $G$, which is 
 known as the {\it spectral radius} of a graph $G$. 
 
\medskip 
There is a rich history on the study of bounding 
the eigenvalues of a graph in terms of 
various parameters; see \cite{Alon1986} for eigenvalues and 
expanders, \cite{Chung1989, CDKL2010} for eigenvalues and diameters, 
\cite{Hong1998} for spectral radius and genus, 
\cite{BZ2001} for spectral radius and cut vertices, 
\cite{C2006, Lu2012} for regularity and eigenvalues, 
\cite{CGN2007, LSW2007} for non-regularity and spectral radius, 
\cite{BN2007jctb} for spectral radius and cliques, 
\cite{Bil2006,WE2013} for chromatic number and eigenvalues, 
\cite{LLT2007, GN2008, Niki2009jctb} for independence number and eigenvalues, 
\cite{CGH2009, OC2010} for matching,  edge-connectivity and eigenvalues, 
\cite{GLLY2016} for spanning trees and eigenvalues, 
\cite{TT2017,LN2021} for eigenvalues of outerplanar and planar graphs, 
and \cite{Tait2019} for the Colin de Verdi\`{e}re parameter, excluded minors and the spectral radius.

\medskip 
Let $G$ be a graph on $n$ vertices with $m$ edges. 
Let $A(G)$ be the adjacency matrix of $G$. 
It is well-known  that 
\begin{equation} \label{eqwell-known}
  \frac{2m}{n} \le \lambda (G)\le \sqrt{2m}. 
  \end{equation}
Indeed, the lower bound is guaranteed by Rayleigh's inequality 
$\lambda (G)\ge \bm{e}^TA(G)\bm{e} = \frac{2m}{n}$, 
where $\bm{e}=\frac{1}{\sqrt{n}}(1,1,\ldots ,1)^T\in \mathbb{R}^n$. 
The upper bound can be seen by invoking the fact that  $\lambda(G)^2 \le \sum_{i=1}^n \lambda_i^2 =\mathrm{tr}(A^2(G)) =\sum_{i=1}^n d_i=2m$.  
This upper bound was  further improved by 
Hong \cite{Hong1988} as 
\begin{equation} \label{eqHong}  \lambda (G) \le \sqrt{2m-n+1}. 
\end{equation}
We recommend the readers to \cite{HSF2001} and 
\cite{Niki2002cpc} for further extensions. 
The classical extremal graph problems 
 usually study the maximum or minimum 
number of edges that the extremal graphs can have. 
Correspondingly, 
we can study the extremal spectral problem.  
We denote by $\mathrm{ex}_{\lambda}(n,F)$ 
 the largest eigenvalue of the adjacency matrix 
in an $n$-vertex  graph that contains no copy of $F$, that is, 
\[ \mathrm{ex}_{\lambda}(n,F):=\max \bigl\{ \lambda(G): |G|=n~\text{and}~F\nsubseteq G \bigr\}. \]

In 1970, Nosal \cite{Nosal1970} 
determined the largest spectral radius of a  triangle-free graph 
in terms of the number of edges, which improved both inequalities 
(\ref{eqwell-known}) and (\ref{eqHong}) and 
also provided 
the spectral version of Mantel's Theorem \ref{thmMan}.  
Note that when we consider the result on a graph with respect to 
the given number of edges, 
we shall ignore the possible isolated vertices if there are no confusions. 

 \begin{theorem}[Nosal, 1970] \label{thmnosal}
Let $G$ be a  graph on $n$ vertices with $m$ edges. 
If $G$ is triangle-free, then 
\begin{equation}   \label{eq1}
\lambda (G)\le \sqrt{m} , 
\end{equation}
 equality holds if and only if 
$G$ is a complete bipartite graph. 
Moreover, we have 
\begin{equation} \label{eq2}
 \lambda (G)\le \lambda (K_{\lfloor \frac{n}{2}\rfloor, \lceil \frac{n}{2} \rceil } ), 
 \end{equation}
equality holds if and only if $G $ 
is a balanced complete bipartite graph 
$K_{\lfloor \frac{n}{2}\rfloor, \lceil \frac{n}{2} \rceil }$. 
\end{theorem}

Nosal's theorem implies that if  $G$ is a  bipartite graph, then 
$  \lambda (G)\le \sqrt{m} $, 
 equality holds if and only if 
$G$ is a complete bipartite graph. 
On the one hand, inequality (\ref{eq1}) implies the classical Mantel Theorem \ref{thmMan}.  
Indeed, applying the Rayleigh inequality, we have 
$\frac{2m}{n}\le \lambda (G)\le  \sqrt{m}$, 
which yields $ m \le \lfloor \frac{n^2}{4} \rfloor$.  
On the other hand,  combining  (\ref{eq1}) with  Mantel's theorem,  we obtain  
$ \lambda (G)\le \sqrt{m}  \le \sqrt{\lfloor {n^2}/{4}\rfloor} =\lambda (K_{\lfloor \frac{n}{2}\rfloor, \lceil \frac{n}{2} \rceil })$. 
So inequality (\ref{eq1}) in Nosal's theorem can imply  (\ref{eq2}). 
Inequality (\ref{eq2}) is called the spectral Mantel theorem.

\medskip 
Nosal's theorem stimulated the developments of two aspects 
in spectral extremal graph theory. 
On the one hand, various extensions and generalizations on 
inequality (\ref{eq2}) in Nosal's theorem have been obtained in the literature; see, e.g., 
\cite{Wil1986,Niki2007laa2,Gui1996,KN2014} for extension 
on $K_{r+1}$-free graphs with given order;  
see \cite{BN2007jctb,Niki2009jctb} for relations 
between cliques and spectral radius 
and \cite{NikifSurvey,CZ2018,LFL2022} for surveys. 
Very recently, Lin, Ning and Wu \cite[Theorem 1.4]{LNW2021}  proved 
a  generalization on (\ref{eq2})  for non-bipartite triangle-free graphs and 
provided a spectral version of Erd\H{o}s' Theorem \ref{thmErd};  
see  \cite{LP2022second} for an 
 alternative proof and  refinement of spectral Tur\'{a}n theorem, 
and  \cite{LP2022} for more stability theorems on spectral graph problems. 
 In addition, Lin and Guo \cite{LG2021} proved 
 an extension of non-bipartite graphs without short odd cycles. 
This result was also independently proved by 
Li, Sun and Yu \cite[Theorem 1.6]{LSY2022} using a different method. 

\medskip 
On the other hand, the inequality (\ref{eq1}) in Nosal's theorem 
boosted the great interests of studying the maximum spectral radius of graphs 
in terms of the number of edges, instead of the number of vertices; 
see \cite{Niki2002cpc} for an extension on $K_{r+1}$-free graphs, 
\cite{Niki2009laa} for an analogue of $C_4$-free graphs, 
\cite{ZLS2021} for further extensions on $K_{2,r+1}$-free graphs,  
and similar results of $C_5$-free and $C_6$-free graphs as well, 
\cite{Niki2021} for an extension on $B_k$-free graphs, 
where $B_k$ denotes the book graph consisting of $k$ triangles 
sharing a common edge, 
and  \cite{LNW2021,ZS2022dm}  for refinements on non-bipartite triangle-free graphs. 
In this paper, we will focus mainly on the extremal spectral problems 
for graphs with given number of edges, which 
is becoming increasingly an important and popular topic 
in recent research on spectral graph theory.

\medskip 

In 2021, Lin, Ning and Wu \cite{LNW2021} 
 proved the following improvement on Nosal's theorem 
 by using tools from doubly stochastic matrix theory; 
 see \cite{Niki2021} for a simpler proof by using 
 elementary numerical inequalities. 
 Let $P_n$ be the path on $n$ vertices, 
  and $C_n$ be the cycle on $n$ vertices. 
 Given two graphs $G$ and $H$, we write $G\cup H$ for the 
disjoint union of $G$ and $H$. In other words, $V(G\cup H)=V(G)\cup V(H)$ 
and $E(G\cup H)=E(G) \cup E(H)$. 
For simplicity, we write $kG$ for the disjoint union of $k$ copies of $G$. 
The blow-up of a graph $G$ is a new graph obtained from 
$G$ by replacing each vertex $v\in V(G)$ with an independent set $I_v$, 
and for two vertices $u,v\in V(G)$, we add all edges between $I_u$ and $I_v$ 
whenever $uv \in E(G)$.

\begin{theorem}[Lin--Ning--Wu, 2021] \label{thm14}
Let $G$ be a triangle-free graph  with $m$ edges. Then 
\[ \lambda_1^2(G) + \lambda_2^2(G) \le m, \]
equality holds if and only if $G$ is a blow-up of a member of $\mathcal{G}$ 
in which 
\[ \mathcal{G} = \{P_2 \cup K_1, 2P_2 \cup K_1, P_4 \cup K_1, P_5 \cup K_1\}.  \]
\end{theorem}

A  conjecture of Bollob\'{a}s and Nikiforov \cite[Conjecture 1]{BN2007jctb} states that if $G$ is a $K_{r+1}$-free graph with $m$ edges, then 
\[  \lambda_1^2(G) + \lambda_2^2(G) \le \left(1-\frac{1}{r} \right)2m. \] 
Theorem \ref{thm14} confirmed the  case 
$r=2$; see \cite{Niki2021,LSY2022} for recent progress. 
This conjecture of Bollob\'{a}s and Nikiforov remains open for the case 
$r\ge 3$. 
We remark here that $\lambda_1^2(G) + \lambda_2^2(G) \le m$ 
does not hold for the $C_4$-free graphs $G$. 
Indeed, take $G=K_{1,m-1}^+$, the graph obtained from the star $K_{1,m-1}$ 
by adding an edge into its independent set. 
For example, 
setting $m=20$, we have $\lambda_1(K_{1,19}^+) \approx 4.425$ 
and $\lambda_2(K_{1,19}^+)=0.890$, while $\lambda_1^2+ \lambda_2^2 \approx 20.372 >20$.

\medskip 
With the help of Theorem \ref{thm14}, 
Lin, Ning and Wu \cite[Theorem 1.3]{LNW2021} further proved the following refinement on (\ref{eq1}) in 
Nosal's theorem for non-bipartite triangle-free graphs with given number of edges.

\begin{theorem}[Lin--Ning--Wu, 2021] \label{thm15}
Let $G$ be a triangle-free graph with $m$ edges.  
If $G$  is non-bipartite, then 
\[  \lambda (G)\le \sqrt{m-1}, \] 
 equality holds if and only if $m=5$ and $G=C_5$. 
\end{theorem}

 In 2022, Zhai and Shu  \cite{ZS2022dm} proved a further improvement 
on Theorem \ref{thm15}. 
Before stating their result, we need to introduce the extremal graph firstly. 
For every integer $m\ge 3$, 
we denote by $\beta(m)$ the largest root of 
\[  Z(x):=x^3-x^2-(m-2)x+m-3. \]  
 It is not difficult to show that   for $m\ge 6$, we have 
 \begin{equation}  \label{eqbeta}
   \sqrt{m-2} < \beta(m) < \sqrt{m-1}. 
   \end{equation}
Furthermore, 
one can verify that $\lim_{m\to \infty} 
(\beta(m)- \sqrt{m-2}) =0$.  
On the other hand, if $m$ is odd, 
let $SK_{2,\frac{m-1}{2}} $ be the graph obtained 
from $K_{2,\frac{m-1}{2}}$ by subdividing an edge; 
see Figure \ref{fig-2} for two different drawings of  $SK_{2,\frac{m-1}{2}} $. 
In particular, for $m=5$, we have $SK_{2,2}=C_5$. 
Clearly, $SK_{2,\frac{m-1}{2}}$ is a triangle-free graph on 
 $n=\frac{m-1}{2} +3$ vertices with  
$m$ edges, and it is non-bipartite as it contains a copy of $C_5$.  
The characteristic polynomial of 
$SK_{2,\frac{m-1}{2}}$ is 
\[  \mathrm{det}(xI_n - A(SK_{2,\frac{m-1}{2}}))=x^{\frac{m-5}{2}} (x^2+x-1) 
\left( x^3-x^2-(m-2)x+m-3\right). \]   
Therefore, if $m$ is odd, then  $\beta(m)$ is
 the largest eigenvalue of  $SK_{2,\frac{m-1}{2}}$. 
 
 For convenience, we denote 
\begin{equation} \label{Hx}
H(x):=(x^2+x-1)Z(x) =x^5 -mx^3 +(2m-5)x - m+3.
 \end{equation} 
So $\beta (m)$ is also the largest root of $H(x)$.

 \begin{figure}[H]
\centering 
\includegraphics[scale=0.8]{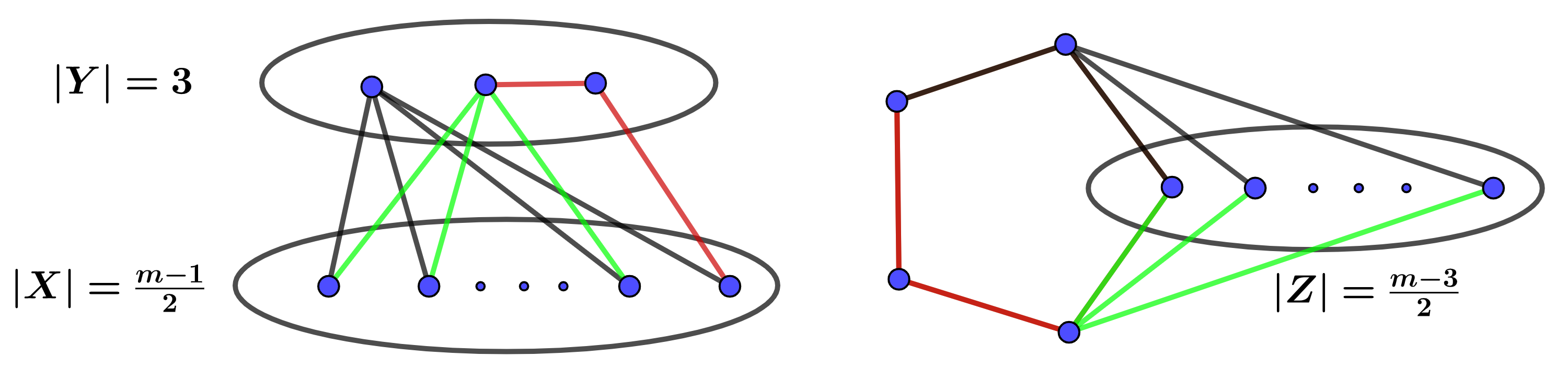}  
\caption{Two  drawings of the graph $SK_{2,\frac{m-1}{2}} $. } \label{fig-2} 
\end{figure}

The improvement of Zhai and Shu \cite{ZS2022dm} on Theorem \ref{thm15} 
can be stated as below. 

\begin{theorem}[Zhai--Shu, 2022] \label{thmZS2022}
Let $G$ be a  graph of size $m$. If $G$ is triangle-free and non-bipartite, then 
\[ \lambda (G) \le \beta (m), \]    
equality holds if and only if $m$ is odd and 
$G=SK_{2,\frac{m-1}{2}}$. 
\end{theorem}

\medskip 

The way that Lin, Ning and Wu \cite{LNW2021} proved 
Theorem \ref{thm15} is original, 
and the line of the proof of Zhai and Shu \cite{ZS2022dm} for Theorem \ref{thmZS2022} 
is technical. 
This paper is organized as follows. 
In Section \ref{sec3}, we shall present an alternative  proof of 
Theorem \ref{thmZS2022}.  The present proof is different from 
the original proof in \cite{ZS2022dm}. 
Our proof uses and develops the ideas in both \cite{LNW2021} and \cite{NZ2021}, 
 we shall make use of the information of all eigenvalues of  graphs, 
instead of the second largest eigenvalue only.  
This proof could introduce the main ideas of the approach of our paper, without some technicalities that arise in the other cases, i.e., it can help us to 
deal with the extremal spectral problem for graphs without short odd cycles. In Section \ref{sec4}, 
by applying the ideas of the proof of Theorem \ref{thmZS2022}, 
we will give further refinement 
on Theorem \ref{thmZS2022}. 
In Section \ref{sec5}, we will 
conclude this paper with some possible open  problems for interested readers.  
This paper can be regarded as a supplement of our previous article \cite{LP2022second}. 
Both of these two papers provide  extensions and generalizations 
on the results involving eigenvalues and triangles.

\section{Alternative proof of Theorem \ref{thmZS2022}} 

\label{sec3}

Recall that Theorem \ref{thmZS2022} 
is an improvement on Theorem \ref{thm15}, 
since $\beta(m) < \sqrt{m-1} $, where $\beta(m)$ is 
the largest root of $x^3-x^2-(m-2)x+m-3=0$.  
The proof of Theorem \ref{thm15}  is succinct and relies on 
Theorem \ref{thm14}, which implies that if 
$G$ is non-bipartite and $\lambda_1^2(G) + \lambda_2^2(G) \ge m$, 
then $G$ contains a triangle. 
Combining the condition in Theorem \ref{thm15}, 
we know that if $G$ satisfies $\lambda_1 (G) \ge \sqrt{m-1}$, 
then $\lambda_2(G)<1$. 
This bound on the second largest eigenvalue 
provided great convenience  to characterize the local structure 
of $G$. For instance,  
combining $\lambda_2(G)<1$ with 
 Cauchy's interlacing theorem, we obtain that 
the shortest odd cycle 
of $G$ is $C_5$. 
However, it is not sufficient to apply 
Theorem \ref{thm14}  for the proof of Theorem \ref{thmZS2022}.   
Indeed, if  $G$ is a graph satisfying $\lambda (G) \ge 
\beta (m)$, 
then invoking the fact that $\lim_{m\to \infty} (\beta(m) 
- \sqrt{m-2}) =0$, 
 we get only that 
$\lambda_2(G) <2$.  
Nevertheless, 
this bound is invalid for our purpose to describe the local structure of $G$. 
The original proof of 
Zhai and Shu \cite{ZS2022dm} for Theorem \ref{thmZS2022} 
is innovative and 
 avoided the use of Theorem \ref{thm14}, thus it 
made more detailed  structure analysis of graphs; 
see \cite{ZS2022dm} for more details. 
 
\medskip 
In what follows, we shall provide an alternative proof 
of Theorem \ref{thmZS2022}.  
Our proof grows out partially from  
the original proof \cite{LNW2021} of Theorem \ref{thm15}.  
To overcome the obstacle mentioned above, 
we shall make full use of the information of all eigenvalues 
of graphs, 
instead of the second largest eigenvalue merely.   
By applying Cauchy's interlacing theorem  of all eigenvalues, 
we will find some forbidden induced subgraphs and 
refine the structure of the desired extremal graph.  
A key idea      
relies on the eigenvalue interlacing theorem and 
 a counting lemma  \cite{NZ2021}, which established the relation 
between eigenvalues and the number of triangles of a graph. 

\medskip 
The main steps of the proof can be outlined as below.  
It introduces the main ideas of the approach of this paper 
for treating the problem involving short odd cycles. 

\begin{itemize}
\item[\ding{73}]  
First of all, applying the forthcoming Lemmas \ref{lemCauchy}, \ref{lem21} and \ref{lem32}, 
we will show that $G$ can not contain the odd cycle $C_{2k+1}$ as an induced subgraph 
for every $k\ge 3$, that is, 
$C_5$ is the shortest odd cycle in $G$; see Claim \ref{claim3.1}.

\item[\ding{78}]  
Upon more computations,  we will prove that  more substructures, e.g., 
 the graphs $H_1,H_2,H_3$ in Figure \ref{fig-3},  are also forbidden 
 as  induced subgraphs in $G$ by applying Lemmas \ref{lemCauchy}, 
 \ref{lem21} and \ref{lem32} again; see Claim \ref{claim3.2}.

\item[\ding{77}]  
Let $S$ be the set of vertices of a copy of $C_5$ in $G$. 
Using the above informations of local structure of $G$, 
we will show that every vertex outside of $S$ has exactly two neighbors 
in $S$; see Claim \ref{claim3.3}. 

\item[\ding{72}]  
Combining with the previous steps, 
we will prove that 
$G$ is isomorphic to the subdivision of 
the complete bipartite graph $K_{a,b}$ by subdividing an edge, 
where $a,b\ge 2$ are integers satisfying $m=ab+1$.  
Finally, we will show that $\lambda (SK_{a,b})$ is at most $\beta (m)$, 
equality holds if and only if $a=2$ or $b=2$. 
\end{itemize}

The following lemma is usually referred to as the eigenvalue interlacing 
theorem, also known as the Cauchy, Poincar\'{e}, or Sturm interlacing theorem. 
It states that the eigenvalues of a principal submatrix of a 
Hermitian matrix interlace those of the underlying matrix; 
see, e.g., \cite[pp. 52--53]{Zhan13} and \cite[pp. 269--271]{Zhang11}. 
It is worth noting that this eigenvalue interlacing theorem 
 provides a useful technique to extremal combinatorics and 
 plays a significant role in two breakthrough works \cite{Huang2019,JTYZZ2021}.

\begin{lemma}[Eigenvalue Interlacing Theorem] \label{lemCauchy}
Let $H$ be an $n \times n$ Hermitian matrix partitioned as 
\[ H= \begin{bmatrix} A & B \\ B^* & C  \end{bmatrix} , \]
where $A$ is an $m\times m$ principal submatrix of $H$ for some 
$m\le n$. Then for every $1\le i\le m$, 
\[  \lambda_{n-m+i}(H) \le \lambda_i (A) \le \lambda_i(H). \]
\end{lemma}

Recall that $t(G)$ denotes the number of triangles in $G$.  
It is well-known that the value of $(i,j)$-entry of $A^k(G)$ 
is equal to the number of walks of length $k$ in $G$ 
starting from vertex $v_i$ 
to $v_j$. 
Since each triangle of $G$ contributes $6$ closed walks of length $3$, 
we can count the number of triangles and obtain 
\begin{equation} \label{count}
t(G)=\frac{1}{6} \sum_{i=1}^n A^3(i,i)= \frac{1}{6}\mathrm{Tr}(A^3)=\frac{1}{6}\sum\limits_{i=1}^n\lambda_i^3.
 \end{equation}

The second lemma needed in this paper
 is a triangle counting lemma  in terms of  
both the eigenvalues and the size of a graph, 
it could be seen from  \cite{NZ2021}. 
This could be viewed as a useful variant of (\ref{count}) by using $\sum_{i=1}^n \lambda_i^2 = \mathrm{tr}(A^2)=
\sum_{i=1}^n d_i=2m$. 

\begin{lemma}  (see \cite{NZ2021}) \label{lem21} 
Let $G$ be a graph on $n$ vertices with $m$ edges. 
If $\lambda_1\ge \lambda_2 \ge \cdots \ge \lambda_n$ 
are all eigenvalues of $G$, then  
\[ t(G)=
\frac{1}{6} \sum_{i=2}^n (\lambda_1 + \lambda_i) \lambda_i^2 + 
\frac{1}{3}(\lambda_1^2-m)\lambda_1.  \] 
\end{lemma}

For convenience, we next introduce a function. 

\begin{lemma} \label{lem32}
Let $f(x) :=(\sqrt{m-2} + x)x^2$. If $a\le x \le b \le 0$, 
then 
\[  f(x)\ge \min\{ f(a),f(b) \}. \]  
\end{lemma}

\begin{proof}
Since $f(x)$ is monotonically increasing when 
$x\in (-\infty, -\frac{2}{3}\sqrt{m-2})$, 
and monotonically decreasing when $x\in [-\frac{2}{3}\sqrt{m-2}, 0]$. 
Thus the desired statement holds immediately. 
\end{proof}

It is the time to show an alternative proof of Theorem \ref{thmZS2022}.

\begin{proof}[{\bf Proof of Theorem \ref{thmZS2022}}]
Suppose that $G$ contains no triangle and $G$ is non-bipartite 
such that $\lambda (G) \ge \beta (m)$. 
We shall prove that $m$ is odd and $G=SK_{2,\frac{m-1}{2}}$.  
Without loss of generality, we may assume that 
$G$ has the maximum value of spectral radius. 
First of all, we can see that 
$G$ must be connected.  
Otherwise, we can choose $G_1$ and $G_2$ as two different components, 
where $G_1$ attains the spectral radius of $G$, 
by identifying two vertices from $G_1$ and $G_2$ respectively, 
we get a new graph with larger  spectral radius, 
which is a contradiction\footnote{
There is another way to get a contradiction.  
We  delete an edge within $G_2$, and then add an edge between $G_1$ and $G_2$. This operation will also lead to a new graph with larger spectral radius.}.  
 It is not hard to verify the desired theorem for $m\le 10$, 
 since we can consider whether $C_7\subseteq G$ or 
 $C_9 \subseteq G$ by a standard case analysis. 
Next, we shall consider the case $m\ge 11$. 
The proof proceeds by the following three claims.

\begin{claim} \label{claim3.1}
The shortest odd cycle in $G$ has length $5$. 
\end{claim}

\begin{proof}[Proof of Claim \ref{claim3.1}]
Since $G$ is non-bipartite, let $s$ be the length of a shortest odd cycle 
in $G$. Since $G$ is triangle-free and non-bipartite, we have $s\ge 5$ and 
$\lambda (G)< \sqrt{m}$ by Theorem \ref{thmnosal}. 
Moreover, a shortest odd cycle $C_s \subseteq G$ 
must be an induced odd cycle. By computation, 
we know that the eigenvalues of $C_s$ are 
$ 2\cos \frac{2\pi k}{s}$, where $k=0,1,\ldots ,s-1$. 
Since $C_s$ is an induced copy in $G$, 
we know that $A(C_s)$ is a principal submatrix of $A(G)$. 
Lemma \ref{lemCauchy} implies that 
$ \lambda_{n-s+i}(G) \le  \lambda_i(C_s) \le \lambda_i (G)$ 
for every $i\in [s]$, 
where $\lambda_i$ means the $i$-th largest eigenvalue. 
We next show that $s=5$. 
We denote by $\lambda_1,\lambda_2,\ldots ,\lambda_n$ 
the eigenvalues of $G$ for simplicity.

Suppose on the contrary that $C_7$ is an induced odd cycle of $G$,
 then $\lambda_2 \ge \lambda_2(C_7) = 2\cos \frac{2\pi}{7}\approx 1.246$  
 and $\lambda_3 \ge \lambda_3(C_7)=2 \cos \frac{12\pi}{7} \approx 1.246$. 
Evidently, we get 
 \[  f(\lambda_2) \ge f(1.246) 
 \ge 1.552 \sqrt{m-2} + 1.934, \]
 and 
  \[  f(\lambda_3) \ge f(1.246)  \ge 1.552 \sqrt{m-2} + 1.934. \]
  Our goal is to get a contradiction by applying Lemma \ref{lem21} and showing $t(G)>0$. 
  It is not sufficient to obtain $t(G)>0$ by using the  positive eigenvalues of $C_7$ only. 
  Next, we are going to consider the negative eigenvalues of $C_7$. 
For $i\in \{4,5,6,7\}$, we know that $\lambda_i(C_7)<0$. 
The Cauchy interlacing theorem yields 
$\lambda_{n-3} \le \lambda_4(C_7)=-0.445$,
$\lambda_{n-2} \le \lambda_5(C_7)=-0.445$, 
$\lambda_{n-1} \le \lambda_6(C_7)=-1.801$ 
and $\lambda_{n} \le \lambda_7(C_7)=-1.801$. 
To apply Lemma \ref{lem32}, we need to 
find the lower bounds on $\lambda_i $ for each 
$i\in \{n-3,n-2,n-1,n\}$. We know from (\ref{eqbeta}) that 
$\lambda_1 \ge \beta (m) > \sqrt{m-2}$, and  then 
$\lambda_n^2 \le 2m - (\lambda_1^2 +\lambda_2^2 + 
\lambda_3^2 + \lambda_{n-3}^2 + \lambda_{n-2}^2 + 
\lambda_{n-1}^2 ) < 2m-(m-2 + 6.744) = m- 4.744$, 
which implies $-\sqrt{m-4.744} < \lambda_n \le -1.801$. By Lemma \ref{lem32}, we get 
\[ f(\lambda_n) \ge 
\min\{f(-\sqrt{m-4.744}), f(-1.801)\}> \sqrt{m-2}  \]
for every $m\ge 11$. Similarly, 
we have $\lambda_{n-1}^2+\lambda_n^2 \le 
2m - (\lambda_1^2 +\lambda_2^2 + 
\lambda_3^2 + \lambda_{n-3}^2 + \lambda_{n-2}^2) 
< m- 1.501$. Combining with $\lambda_{n-1}^2 \le \lambda_n^2$, 
we get $-\sqrt{(m-1.501)/2} < \lambda_{n-1} \le -1.801$. 
By Lemma \ref{lem32}, we obtain 
\[  f(\lambda_{n-1}) \ge 
\min \{f(-\sqrt{(m-1.501)/2}), f(-1.801)\} > \sqrt{m-2}  \]
for every $m\ge 9$. 
Note that $
\sqrt{m} > \lambda_1  \ge \beta (m) > \sqrt{m-2}$. 
By Lemma \ref{lem21}, we get 
\begin{align*} 
t(G)  &> \frac{1}{6} (f(\lambda_2) + f(\lambda_3) + f(\lambda_n) 
+ f(\lambda_{n-1}) )
 - \frac{2}{3}\lambda_1 \\
 & >\frac{1}{6} (5.104 \sqrt{m-2} - 4\sqrt{m}+ 3.868) >0 . 
\end{align*} 
This is a contradiction. Similarly, we can prove 
by applying the monotonicity of $\cos x$ that 
$C_s$ can not be an induced subgraph of $G$ for each odd integer $s\ge 7$. 
Thus we get $s=5$. 
\end{proof}

Let $S=\{u_1,u_2,u_3,u_4,u_5\}$ be the set of vertices of 
 a copy of $C_5$ in $G$.  We define the graphs $H_1,H_2$ 
 and $H_3$ as in Figure \ref{fig-3}. 
 The eigenvalues of these graphs can be seen in Table \ref{tab-1}. 
 To avoid unnecessary calculations, we did not attempt to get the best bound on the size of graph in the  proof. 
Next, we  consider the case $m\ge 514$  
in the remaining proof.

 \begin{figure}[H]
\centering 
\includegraphics[scale=0.75]{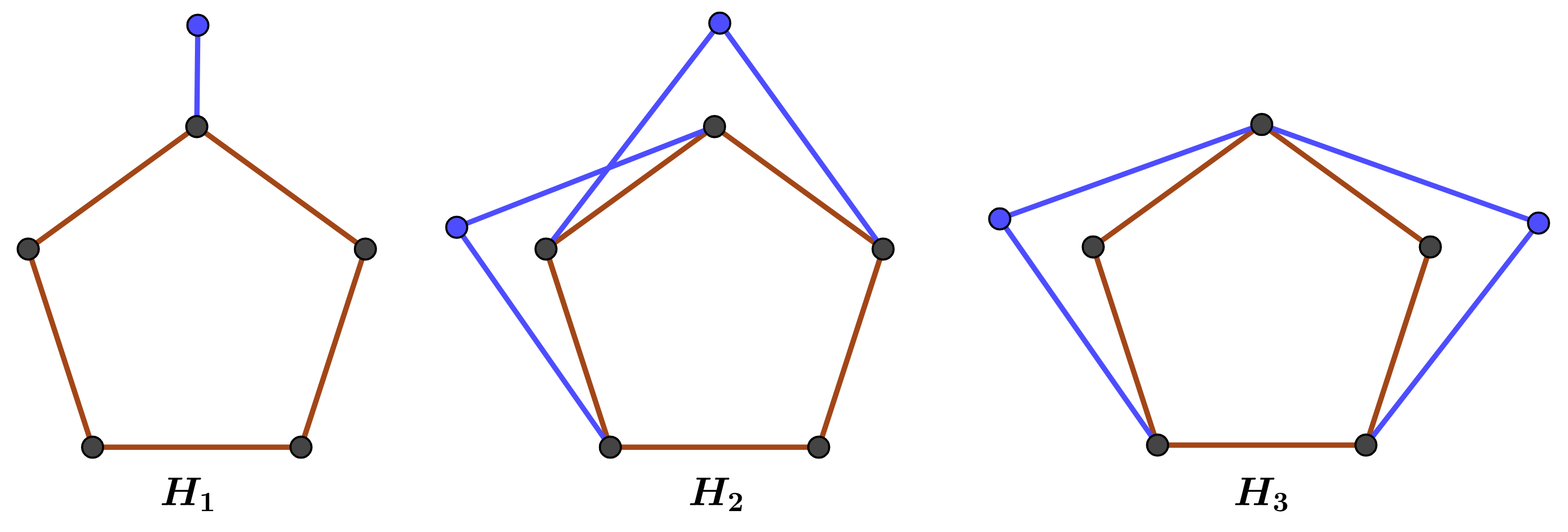} 
\caption{Some forbidden induced subgraphs in $G$. } \label{fig-3} 
\end{figure}

\begin{table}[H]
\centering 
\begin{tabular}{cccccccc}
\toprule
    & $\lambda_1$  & $\lambda_2$  &  $\lambda_3$ 
    &  $\lambda_4$ &  $\lambda_5$ &  $\lambda_6$ & $\lambda_7$ \\ 
\midrule
 $C_7$ & 2 & 1.246 & 1.246 & $-0.445$ & $-0.445$ & $-1.801$ &
 $-1.801$ \\ 
$H_1$ & 2.115  &  1 & 0.618 & $-0.254$ & $-1.618$ & $-1.860$ 
 \\ 
$H_2$ & 2.641  &  1 &  0.723 & 0.414 & $-0.589$ & $-1.775$ & $-2.414$  \\ 
$H_3$ & 2.681 & 1 & 0.642 & 0 & 0 & $-2$ & $-2.323$   \\ 
\bottomrule 
\end{tabular} 
\caption{Eigenvalues of forbidden induced subgraphs.} \label{tab-1}
\end{table}

\begin{claim} \label{claim3.2}
 $G$ does not contain any graph of 
 $\{H_1,H_2,H_3\}$ as an induced 
 subgraph. 
 \end{claim}
 
 \begin{proof}[Proof of Claim \ref{claim3.2}]
 Suppose on the contrary that $G$ contains $H_i$ 
 as an induced subgraph for some $i\in \{1,2,3\}$. 
 To obtain a contradiction, we shall show $t(G)>0$ by applying 
 Lemma \ref{lem21}. 
 We first consider the case that $H_1$ is an induced subgraph 
 in $G$. 
 The Cauchy interlacing theorem implies 
 $\lambda_{n-6+i}(G) \le \lambda_i(H_1) \le \lambda_i (G)
 $ for every $i\in [6]$. 
 We denote $\lambda_i=\lambda_i(G)$ for short. 
Obviously, we have 
 \[  f(\lambda_2 ) \ge f(1)=\sqrt{m-2} +1, \]
 and 
 \[ f(\lambda_3 ) \ge f(0.618)\ge 0.381\sqrt{m-2} + 0.236.  \]
 We next consider the negative eigenvalues of $G$. 
 The Cauchy interlacing theorem implies 
 $\lambda_{n-2} \le \lambda_4(H_1)=-0.254$ 
 and $\lambda_{n-1} \le \lambda_5(H_1)=-1.618$ 
 and $\lambda_n \le \lambda_6 (H_1)=-1.860$. 
Moreover, we get from (\ref{eqbeta}) that 
$\lambda_1 \ge \beta (m) > \sqrt{m-2}$ and 
$\lambda_n^2  
 \le 2m- (\lambda_1^2 + \lambda_2^2 + 
 \lambda_3^2  + \lambda_{n-2}^2 + \lambda_{n-1}^2 ) 
 \le 2m-(m-2+4.064)=m-2.064$, which implies $-\sqrt{m-2.064} < \lambda_n \le 
 -1.860$. By Lemma \ref{lem32}, we have 
 \[  f(\lambda_n ) \ge 
 \min\{ f(-\sqrt{m-2.064}), f(-1.860)\} > 0.031 \sqrt{m-2}. \]
Secondly, since $\lambda_{n-1}^2 + \lambda_n^2  
\le 2m- (\lambda_1^2 +\lambda_2^2 + 
\lambda_3^2 + \lambda_{n-2}^2 ) 
< m+0.553$ and $\lambda_{n-1}^2 \le \lambda_n^2$, 
we get $-\sqrt{(m+0.553)/2}< \lambda_{n-1} \le -1.618$. 
 By Lemma \ref{lem32},  we get 
 \[  f(\lambda_{n-1}) \ge 
 \min\{ f(-\sqrt{(m+0.553)/{2}}), f(-1.618)\} > 2.617 \sqrt{m-2} - 4.235 \] 
 for every $m\ge 12$. 
 Moreover, we have $-\sqrt{{(m+0.618)}/{3}}<\lambda_{n-2}\le -0.254$ and then 
 \[ f(\lambda_{n-2}) \ge \min \{f(-\sqrt{({m+0.618})/{3}}), 
 f(-0.254)\} >0.064\sqrt{m-2} -0.016  \]
 for every $m\ge 4$. 
 By Lemma \ref{lem21}, we get that for  $m\ge 514$, 
\begin{align*} 
t(G)  &> \frac{1}{6}(f(\lambda_2) + f(\lambda_3) + f(\lambda_{n-2})
+f(\lambda_{n-1}) + f(\lambda_n )) - \frac{2}{3} \lambda_1 \\ 
&>\frac{1}{6}(4.093\sqrt{m-2}-4\sqrt{m} -2.015)>0,
\end{align*}
which is a contradiction.

If $H_2$ is an induced subgraph of $G$,  
 then we get similarly that
$\lambda_2 \ge 1, \lambda_3\ge 0.723$ and 
$\lambda_4 \ge 0.414$. Then 
\begin{gather*} 
f(\lambda_2) \ge f(1)= \sqrt{m-2} +1, \\ 
 f(\lambda_3) \ge f(0.723) \ge 0.522 \sqrt{m-2} +0.377, 
\end{gather*}
and 
\[ f(\lambda_4) \ge f(0.414)\ge 0.171\sqrt{m-2} + 0.07.  \]
The negative eigenvalues of $H_2$ imply that 
$\lambda_{n-2}\le -0.589$, $\lambda_{n-1}\le -1.775$ 
and $\lambda_n \le -2.414$. 
Due to $\lambda_n^2 \le 2m- (\lambda_1^2+\lambda_2^2 + \lambda_3^2 + 
\lambda_4^2 + \lambda_{n-2}^2 +\lambda_{n-1}^2) 
< 2m - (m-2 + 5.191)=m-3.191$, 
 we  get  $-\sqrt{m-3.191} \le \lambda_n \le -2.414$.   
 Lemma \ref{lem32} gives 
\[  f(\lambda_n) \ge \min\{ f(-\sqrt{m-3.191}), f(-2.414)\} 
> 0.5\sqrt{m-2} \]
for every $m\ge8$. In addition, 
we have $-\sqrt{({m-0.041})/{2}} \le \lambda_{n-1} \le -1.775$ and 
\[ f(\lambda_{n-1}) \ge \min\{ f(-\sqrt{{(m-0.041)}/{2}}), f(-1.775)\} 
> 2\sqrt{m-2}  \]
for every $m\ge 17$. 
By Lemma \ref{lem21}, we obtain that for $m\ge 6$,  
\begin{align*} 
t(G)  &> \frac{1}{6}(f(\lambda_2) + f(\lambda_3) + f(\lambda_4) + 
f(\lambda_{n-1}) + f(\lambda_n )) - \frac{2}{3} \lambda_1 \\
&> \frac{1}{6}(4.193\sqrt{m-2}-4\sqrt{m} +1.447)>0, 
\end{align*}
which is also a contradiction.

If $H_3$ is an induced subgraph of $G$, 
 then we get $\lambda_2\ge 2$ and $\lambda_3\ge 0.642$. 
 Then 
\[  f(\lambda_2) \ge f(1)=\sqrt{m-2} +1 \]
and 
\[  f(\lambda_3) \ge f(0.642) \ge 0.412\sqrt{m-2} +0.264. \] 
Moreover, Cauchy's interlacing theorem gives $\lambda_{n-1}\le -2$ 
and $\lambda_n \le -2.323$. 
Since $\lambda_n^2\le 2m-(\lambda_1^2+\lambda_2^2+\lambda_3^2 + \lambda_{n-1}^2)< 2m-(m-2 +5.412)=m-3.412$, we get 
$-\sqrt{m-3.412} < \lambda_n \le -2.323$. Then 
\[  f(\lambda_n) \ge \min\{ f(-\sqrt{m-3.412}), f(-2.323)\} \ge 
0.7\sqrt{m-2}.  \]
Similarly,  
we have $-\sqrt{(m+0.587)/2} < \lambda_{n-1}\le -2$ and 
\[  f(\lambda_{n-1}) \ge \min\{ f(-\sqrt{(m+0.587)/2}), f(-2)\} 
\ge 4\sqrt{m-2} -8. \]
By Lemma \ref{lem21}, we obtain 
\begin{align*} 
t(G)  &> \frac{1}{6}(f(\lambda_2) + f(\lambda_3) + 
f(\lambda_{n-1}) + f(\lambda_n )) - \frac{2}{3} \lambda_1 \\
&> \frac{1}{6}( 6.112\sqrt{m-2}-4\sqrt{m} -6.736 )>0, 
\end{align*}
 which is a contradiction. 
 \end{proof}
 
 Let $N(S):=\cup_{u\in S} N(u)$ be the union of 
 neighborhoods of vertices of  $S$. We denote by $d_S(v)=|N(v)\cap S|$ the 
 number of neighbors of $v$ in the set $S$. 
 
 \begin{claim} \label{claim3.3}
 $V(G)=S \cup N(S)$ and 
 $d_S(v)=2$ for every $v\in N(S)$. 
 \end{claim}
 
 \begin{proof}[Proof of Claim \ref{claim3.3}]
First of all, we prove that $d_S(v)=2$ for each vertex $v\in N(S)$. 
Without loss of generality, 
 we may assume that $v\in N(u_1)$. 
 If $d_S(v)\ge 3$, then  there exists $i\in [5]$ 
 such that $\{v,u_i,u_{i+1}\}$ 
 forms a triangle in $G$, a contradiction. 
 If $d_S(v)=1$, then $S\cup \{v\}$ 
 induces a copy of $H_1$, a contradiction. 
 This implies that $d_S(v)=2$ for every $v\in N(S)$.  
 Next we prove that $V(G)=S\cup N(S)$. 
 Otherwise, if there is a vertex $v' \in V(G) \setminus (S\cup N(S))$, then $v'$ has distance at least $2$ from $S$. 
 We may assume that $v'vu_1$ is an induced $P_3$ such that 
 $v'u_i \notin E(G)$ for every $i\in [5]$. 
 From the above discussion,  we know from $vu_1\in E(G)$ that 
 $d_S(v)=2$ . 
 By symmetry, we may assume that $N_S(v)=\{u_1,u_3\}$. 
 Since $G$ is triangle-free and $v'u_i \notin E(G)$ 
 for every $i\in [5]$, we can see that 
 $\{v',v,u_3,u_4,u_5,u_1\}$ induces a copy of $H_1$, 
 a contradiction. 
 Thus, we conclude that $V(G)=S \cup N(S)$ and 
 $d_S(v)=2$ for every $v\in N(S)$. 
 \end{proof}
 
 Since $m\ge 11$, we can fix a vertex $v\in N(S)$ and 
assume that $N_S(v)=\{u_1,u_3\}$. 
For each $w\in V(G)\setminus (S\cup \{v\})$, 
since $G$ contains no triangles and  no  
$H_3$ as an induced subgraph, we know that 
$N_S(w)\neq \{u_3,u_5\}$ and $N_S(w)\neq \{u_4,u_1\}$. 
It is possible that $N_S(w)=\{u_1,u_3\},\{u_2,u_4\}$ or 
$\{u_5,u_2\}$. Furthermore, 
if $N_S(w)=\{u_1,u_3\}$, then $wv \notin E(G)$, since $G$  contains no  triangles; 
if $N_S(w)=\{u_2,u_4\}$, then 
$wv \in E(G)$, since $G$ contains no induced copy of 
$H_2$. 
We denote $N_{i,j}=\{w\in V(G)\setminus S : N_S(w)=\{u_i,u_j\}\}$. Note that $G$ has no induced copy of $H_3$, 
 there are at least one empty set 
in $\{N_{2,4},N_{5,2}\}$. 
 If $N_{2,4}=\varnothing$ and $N_{5,2}=\varnothing$, 
then $V(G)\setminus S = N_{1,3}$. Thus  $m$ is odd and 
$G=SK_{2,\frac{m-1}{2}}$. 
Without loss of generality, 
if  $N_{2,4} \neq \varnothing$, then 
$V(G)\setminus S= N_{1,3} \cup N_{2,4}$. 
Moreover, $N_{1,3}$ and $N_{2,4}$ induce a complete 
bipartite subgraph in $G$. 
We denote $A=N_{1,3} \cup \{u_2,u_4\}$ 
and $B=N_{2,4}\cup \{u_3,u_1\}$. 
Clearly, we have $|A|=a \ge 2$ and $|B|=b\ge 2$.  
Then we  observe that 
$G$ is isomorphic to the subdivision of 
the complete bipartite graph $K_{a,b}$ by 
subdividing the edge $u_1u_4$ of $K_{a,b}$, 
and $m=e(G)= ab +1$. 
By a direct computation, we  get that 
$\lambda(G)\le \beta(m)$, equality holds if and only if 
$a=2$ or $b=2$, and thus $m$ is odd and $G=SK_{2,\frac{m-1}{2}}$.   The detailed computations are stated below.
The characteristic polynomial of $G=SK_{a,b}$ is 
\begin{align*}
&  \mathrm{det}(xI_n - A(SK_{a,b})) \\
&= x^{a+b-4}
\left(x^5 - (ab+1) x^3+ (3ab -2a -2b +1)x -2ab +2a +2b -2 \right). 
\end{align*}  
Hence $\lambda(G)$ is the largest root of 
\[  F(x):=x^5-m x^3+ (3m-2 -2a -2\tfrac{m-1}{a})x -2m +2a +2\tfrac{m-1}{a}.\]   
Recall in (\ref{Hx}) that $\beta(m)$ denotes the largest root of 
$ H(x)$.  
We can easily verify that 
\[ H(x)- F(x)=(2a+2\tfrac{m-1}{a} -m-3)(x-1),  \]
which yields $H(x)\le F(x)$ for every $x\ge 1$. Then we get $H(\lambda (G)) \le F(\lambda (G))=0$, 
which implies $\lambda (G) \le \beta(m)$. This completes the proof. 
\end{proof}

\noindent 
{\bf Remark.} 
The Nosal Theorem \ref{thmnosal} 
asserts that if $G$ is a graph with $\lambda (G) \ge \sqrt{m}$, 
then either $G$ contains a triangle, or $G$ is a complete bipartite graph. 
Very recently, Ning and Zhai \cite{NZ2021} proved an elegant spectral counting result, 
which states that if $G$ is an $m$-edge graph with 
$\lambda (G) \ge \sqrt{m}$, then $G$ has at least $\lfloor \frac{\sqrt{m} -1}{2}\rfloor $ 
triangles, unless $G$ is a   complete bipartite graph. 
Clearly, this saturation result is a generalization of Nosal's theorem 
as well as a spectral analogue of a result of Rademacher. 
A natural question is whether the counting result analogous to Theorem \ref{thmZS2022} is true. 
More precisely, if $G$ is non-bipartite with $\lambda (G) \ge \beta (m)$, 
then it seems possible that $G$ has at least $\Omega(\sqrt{m})$ triangles, 
unless $G=SK_{2,\frac{m-1}{2}}$. 

Although we can see from the proof of Theorem \ref{thmZS2022}  that 
many cases can yield the conclusion that $G$ has at least $\Omega (\sqrt{m})$ triangles, 
the answer for the above question is surprisingly {\sc negative}. 
Taking $G=K_{1,m-1}^+$ as the graph obtained from the star 
$K_{1,m-1}$ by adding an edge into its independent set, 
we can see that $G$ is not bipartite and 
$\lambda (K_{1,m-1}^+) > \sqrt{m-1} >\beta (m)$, while 
$G$ has only one triangle and $G\neq SK_{2,\frac{m-1}{2}}$. 
Note that the graph $K_{1,m-1}^+$ has $m$ edges on $m$ vertices. 
Moreover, we can show that 
$\lambda (K_{1,m-1}^+)$ is the largest root of the equation 
\[   x^3 -x^2 - (m-1)x + m-3=0.  \]  
For $m=4,5,6,7,8$, we can verify that $\lambda (K_{1,m-1}^+)> \sqrt{m}$;  
while for $m=9$, we get $\lambda (K_{1,8}^+)=3=\lambda(K_{1,9})$. 
For $m\ge 11$, we can check that 
$\lambda (K_{1,m-1}^+) < \sqrt{m}$.

\section{Graphs without short odd cycles}

\label{sec4}

Let $S_{3}(K_{a,b})$ denote the 
graph obtained from the complete bipartite graph 
$K_{a,b}$ by replacing an edge with a five-vertex path $P_{5}$, 
that is, introducing three new vertices on an edge. 
Clearly,  the shortest odd cycle in $S_{3}(K_{a,b})$ has length seven. 

We next consider the further extension of Theorem \ref{thmZS2022} 
for graphs with given size and no short odd cycles. 
For each integer $m\ge 7$, we denote by $\gamma (m)$ 
the largest root of 
\begin{equation} \label{Lx}
 L(x):=x^7 - mx^5 + (4m-14)x^3 - (3m-14)x -m +5. 
 \end{equation}
It is not difficult to check that 
\begin{equation}  \label{eq-gamma}
 \sqrt{m-4} < \gamma (m) \le \sqrt{m-3}. 
 \end{equation}
 Indeed, we observe that 
\begin{align*} 
 L(\sqrt{m-4}) < x(x^6 - mx^4 + (4m-14)x^2 - (3m-14)) \big|_{x=\sqrt{m-4}} = 
 -m+6 \le 0, 
 \end{align*}
 which leads to $\sqrt{m-4} < \gamma (m)$. 
For every $m\ge 7$, we have 
\[   L(\sqrt{m-3}) = \sqrt{m-3}(m(m-11) +29) -m +5 \ge 0,\]
 equality 
holds only for $m=7$. 
Combining with $L'(x)=7x^6-5mx^4 +3(4m-14)x^2 -(3m-14) \ge 0$ 
for every $x\ge \sqrt{m-3}$, 
we get  $L(x)\ge L(\sqrt{m-3})\ge 0$ for every $x\ge \sqrt{m-3}$, 
which implies 
$\gamma (m)\le \sqrt{m-3}$.

Moreover, if $m$ is odd, 
let $S_3(K_{2,\frac{m-3}{2}})$ be the graph obtained from 
the complete bipartite graph $K_{2,\frac{m-3}{2}}$ 
by subdividing an edge into a path of length $4$, i.e., 
putting $3$ new vertices on an edge; see Figure \ref{fig-4}. 
In particular, for $m=7$, we have $S_3(K_{2,2})=C_7$. 
Clearly, $S_3(K_{2,\frac{m-3}{2}})$ has $n=\frac{m-3}{2} +5$ vertices 
and $m$ edges. 
Moreover, 
$S_3(K_{2,\frac{m-3}{2}})$ contains no copy of 
both $C_3$ and $C_5$, but it has a copy of $C_7$ and so it is non-bipartite. 
Upon computation, 
the characteristic polynomial of $S_3(K_{2,\frac{m-3}{2}})$ is given as 
\[  \det (xI_n - A(S_3(K_{2,\frac{m-3}{2}}))) 
= x^{\frac{m-7}{2}} \bigl( x^7 - mx^5 + (4m-14)x^3 - (3m-14)x -m +5
\bigr).  \]
Hence, if $m$ is odd, then $\gamma (m)$ is the largest eigenvalue of 
$S_3(K_{2,\frac{m-3}{2}})$. 

 \begin{figure}[H]
\centering 
\includegraphics[scale=0.8]{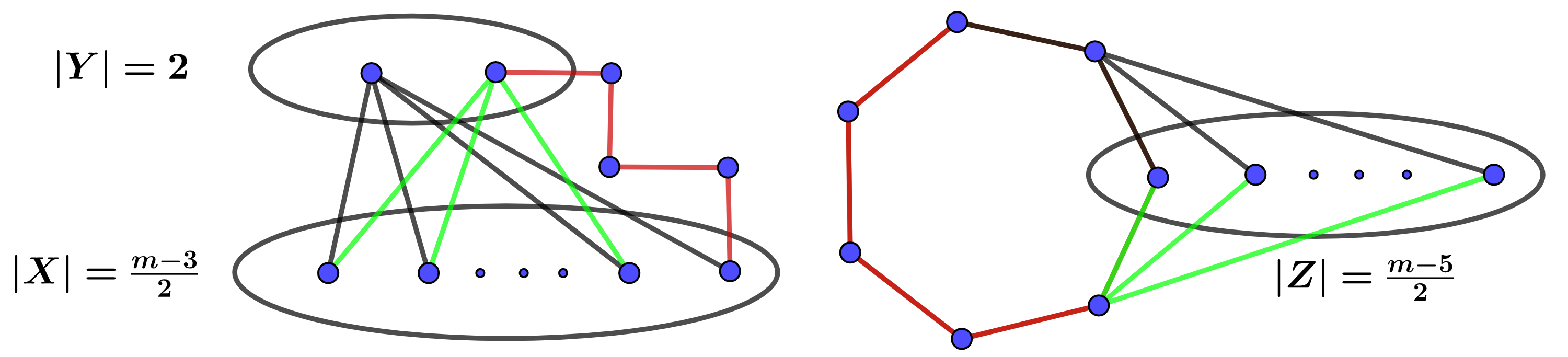}  
\caption{Two  drawings of the graph $SK_{2,\frac{m-3}{2}}$.} \label{fig-4} 
\end{figure}

Note that the extremal graph $SK_{2,\frac{m-1}{2}}$ in Theorem \ref{thmZS2022} contains a copy of $C_5$. 
In this section, we will prove a refinement on Theorem \ref{thmZS2022}. 
To be more specific, we will determine the largest spectral radius 
for  $C_3$-free and $C_5$-free non-bipartite graphs. 
To proceed, we need to introduce a lemma.

\begin{lemma}  \label{lem42}
Let $a,b\ge 2$ and $m$ be integers with $m=ab+4$. 
If $G$ is one of the $m$-edge graphs  obtained from 
$S_3(K_{a,b})$ by adding an edge  to one vertex, then 
$ \lambda (G) < \gamma (m)$. 
\end{lemma}

 \begin{figure}[H]
\centering 
\includegraphics[scale=0.8]{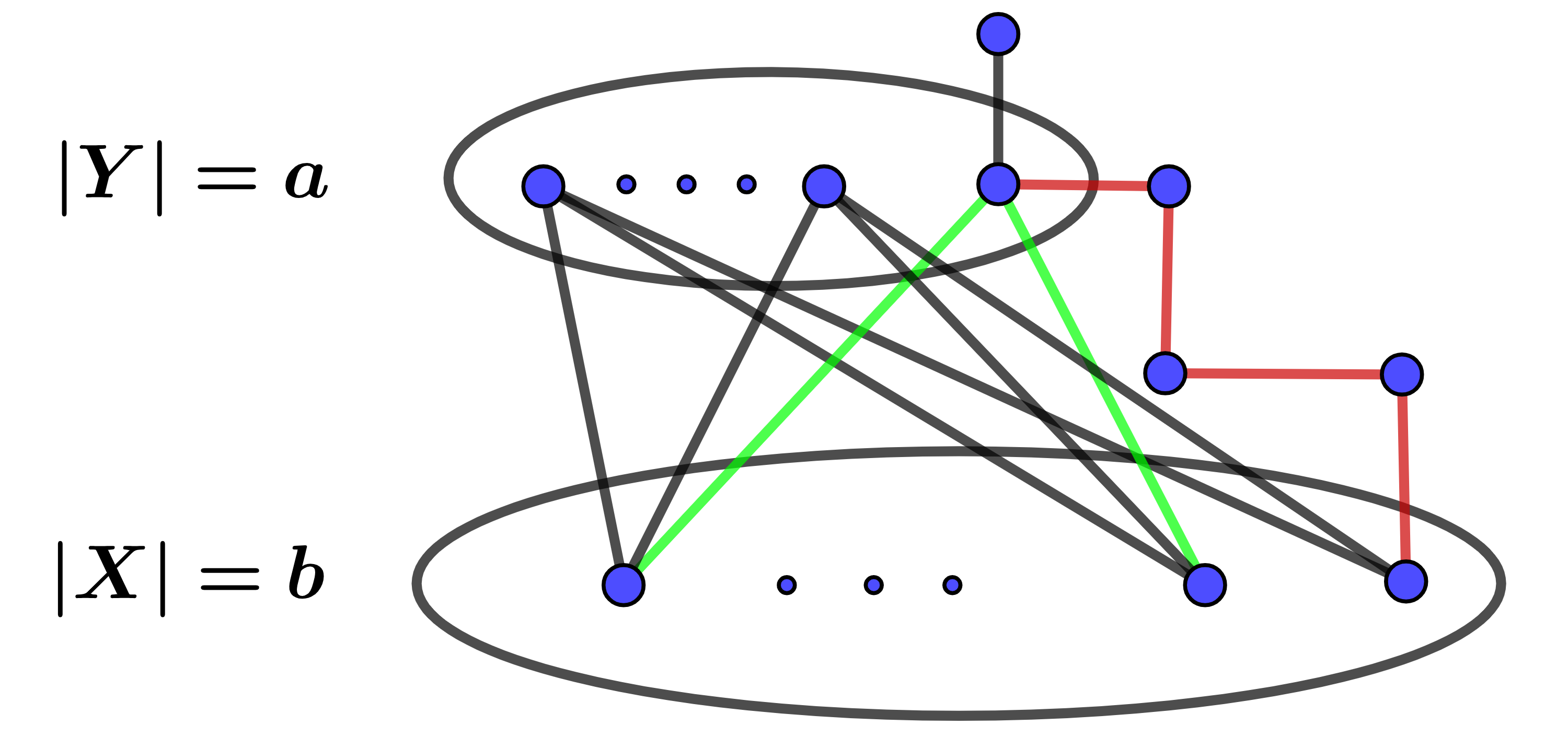}  
\caption{A possible graph in Lemma \ref{lem42}.} \label{fig-only} 
\end{figure} 

\begin{proof}
We know that  $G$ has $7$ possible cases.  
We prove the above case in Figure \ref{fig-only} 
only, since the other cases can be proved in  the same way. 
By computation, we obtain that 
$\lambda (G)$ is the largest root of $E(x)$, where $E(x)$ is defined as 
\begin{align*} 
E(x)&:=x^8 - (ab+4)x^6 + (6ab -2a-3b+5)x^4 -( 8ab-5a-7b+5)x^2 \\
&\quad - (2ab-2a-2b+2)x + ab - a - b + 1. 
\end{align*}
Note that $m=ab+4$ and 
\begin{align*}
E(x)-xL(x) &= (2a-3)(b-1)x^4 -(5a-7)(b-1)x^2 \\
&\quad -(ab-2a-2b+3)x + ab-a-b+1.
\end{align*}
We can verify that $x^2L(x) < E(x)$ 
for every $x\ge \sqrt{(a-1)b}$. 
Since $K_{a-1,b}$ is a subgraph of $G$, we know that 
$\lambda (G)\ge \lambda (K_{a-1,b})=\sqrt{(a-1)b}$. 
Then $\lambda^2(G) L(\lambda (G)) < E(\lambda (G))=0$, 
which yields $L(\lambda (G))<0$ and $\lambda (G)<\gamma (m)$. 
\end{proof}

The main result of this section is as follows. 

\begin{theorem} \label{thmmain}
Let $G$ be a graph with $m$ edges.  
If $G$ does not contain any member of $\{C_3,C_5\}$
 and $G$ is non-bipartite, 
then 
\[  \lambda (G) \le \gamma (m),  \]
equality holds if and only if $m$ is odd and 
$G=S_{3}(K_{2, \frac{m-3}{2} })$.  
\end{theorem}

\begin{proof} 
Assume that $G$ has no $C_3$ and $C_5$, 
and $G$ is non-bipartite 
with $\lambda (G) \ge \gamma (m)$, we will show that $m$ is odd and 
$G=S_3(K_{2,\frac{m-3}{2}})$. 
Similar with that in the proof of Theorem \ref{thmZS2022} 
in Section \ref{sec3}, 
it is sufficient to consider the case that $G$ is connected. 
Since $G$ is non-bipartite, 
we can assume that $C_s$ is a shortest odd cycle of $G$. 
Note that a shortest odd cycle in $G$ must be an induced subgraph.  
Since $G$ is $C_3$-free and $C_5$-free, we have $s\ge 7$ and 
$\lambda (G) < \sqrt{m}$  by Theorem \ref{thmnosal}.  In what follows, we shall show that $s=7$. 
We denote 
\[  g(x):=(\sqrt{m-4} +x)x^2.\]   
Let $\lambda_1 ,\lambda_2,\ldots ,\lambda_n$ be the eigenvalues of $G$ in decreasing order. 
Since $G$ is non-bipartite and $G$ has an induced odd cycle 
 of length at least $7$.  
For $m\le 10$, we can do few case analysis 
whether $C_7\subseteq G$ or $C_9\subseteq G$. 
Thus it is easy to verify the required theorem for $m\le 10$. 
Next, we shall consider the case $m\ge 11$ 
in the  proof.   

\begin{claim} \label{claim4.1}
A shortest odd cycle in $G$ is $C_7$. 
\end{claim}

\begin{proof}[Proof of Claim \ref{claim4.1}]
Assume  that $C_9$ is an induced odd cycle in $G$, 
the Cauchy interlacing theorem implies 
$\lambda_{n-9+i}(G) \le \lambda_i(C_9) \le \lambda_i(G)$ 
for every $i\in \{1,2,\ldots ,9\}$. 
From the following Table 2, we can see that 
$\lambda_2,\lambda_3 \ge 1.532$. Then 
\[ g(\lambda_2) , g(\lambda_3) \ge g(1.532) \ge 2.347 \sqrt{m-4} + 3.596. \]
Moreover, we have $\lambda_4,\lambda_5 \ge 0.347$, 
which implies 
\[ g(\lambda_4), g(\lambda_5) \ge g(0.347)\ge 0.12 \sqrt{m-4} + 0.041.  \]
We next consider the negative eigenvalues of $G$. 
Note from (\ref{eq-gamma}) that $\lambda_1 \ge \gamma (m)>\sqrt{m-4}$. 
Since $\lambda_{n-3}\le \lambda_{n-2} \le \lambda_7(C_9)=-1$ 
and $\lambda_n \le \lambda_{n-1} \le \lambda_8(C_9)=-1.879$,  
we have $\lambda_n^2 \le 2m- (\sum_{i=1}^5 \lambda_i^2 + \lambda_{n-3}^2 + \lambda_{n-2}^2 + \lambda_{n-1}^2)<
2m-(m-4+10.465)= m-6.465$. 
Thus $-\sqrt{m-6.465} \le \lambda_n \le -1.879$ and then 
\[  g(\lambda_n) \ge 
\min\{ g(-\sqrt{m-6.465}), g(-1.879)\} > 0.8 \sqrt{m-4}, \]
where the last inequality holds for every $m\ge 11$. 
Similarly, we have $\lambda_{n-1}^2 + \lambda_n^2 \le 
2m - (\sum_{i=1}^5 \lambda_i^2 + \lambda_{n-3}^2+\lambda_{n-2}^2) 
< m-2.934$, which together with $\lambda_{n-1}^2 \le \lambda_n^2$ 
yields $-\sqrt{{(m-2.934)}/{2}} < \lambda_{n-1} \le -1.879$. Then 
for every $m\ge 11$, we have 
\[  g(\lambda_{n-1}) \ge \min\{g(-\sqrt{{(m-2.934)}/{2}}) , g(-1.879)\} 
> 0.9 \sqrt{m-4}. \]
Moreover, we can similarly get $-\sqrt{{(m-1.934)}/{3}} < \lambda_{n-2} \le -1$ and 
\[ g(\lambda_{n-2}) \ge \min\{g(-\sqrt{{(m-1.934)}/{3}}), g(-1)\} 
\ge \sqrt{m-4} -1. \]
The inequality $-\sqrt{{(m-0.934)}/{4}} < \lambda_{n-3} \le -1$  implies 
\[ g(\lambda_{n-3}) \ge \min\{g(-\sqrt{{(m-0.934)}/{4}}), g(-1)\} 
\ge \sqrt{m-4} -1. \]
Owing to $\sqrt{m}\ge \lambda (G)\ge \gamma (m)>\sqrt{m-4}$,  
by Lemma \ref{lem21}, we obtain  
\begin{equation*}
\begin{aligned} 
 t(G) &> \frac{1}{6}
 \left(\sum_{i=2}^5 g(\lambda_i) + g(\lambda_{n-5+i}) \right)  - \frac{4}{3} \lambda_1(G) \\
& > \frac{1}{6}\left( 8.634\sqrt{m-4} - 8\sqrt{m} + 5.274 \right) >0, 
\end{aligned}
\end{equation*}
which is a contradiction. Therefore, the odd cycle 
$C_9$ can not be an induced subgraph in $G$. Similarly, we can show 
by using the monotonicity of $\cos x$ 
 that $C_s$ is not an induced subgraph of $G$ for each 
$s\ge 11$. Consequently, we get $s=7$.  
\end{proof}

 From Claim \ref{claim4.1}, we denote by $S=\{u_1,u_2,\ldots ,u_7\}$ 
the set of vertices of a copy of  $C_7$ in $G$.  
Next, we shall show that the following graphs are forbidden induced subgraphs in $G$, 
and compute their eigenvalues; see Figure \ref{fig-5} and Table \ref{tab-2}.

 \begin{figure}[H]
\centering 
\includegraphics[scale=0.7]{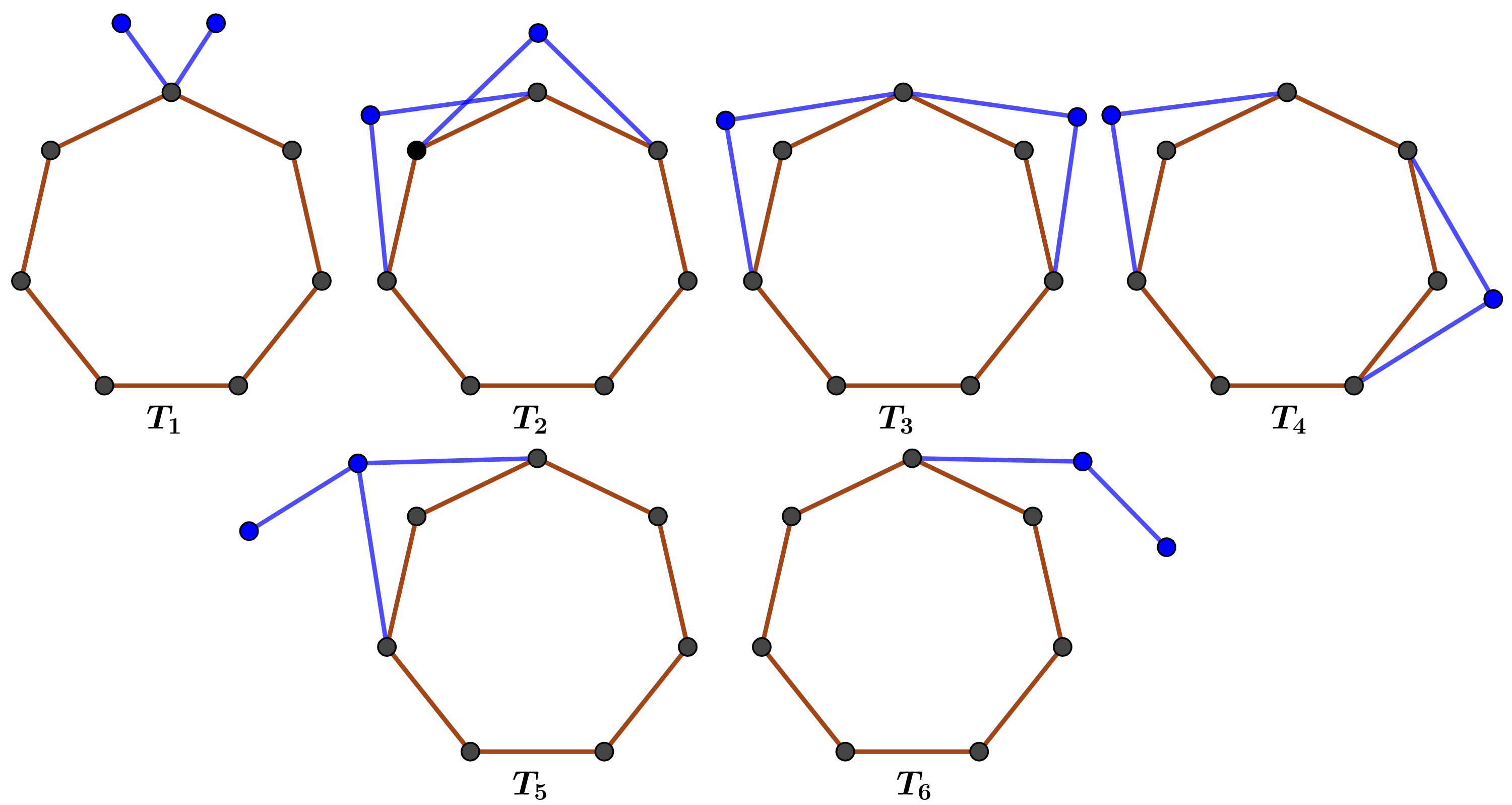}  
\caption{Some forbidden induced subgraphs in $G$.} \label{fig-5} 
\end{figure} 

\begin{table}[H]
\centering  
\begin{tabular}{cccccccccc}
\toprule 
    & $\lambda_1$  & $\lambda_2$  &  $\lambda_3$ 
    &  $\lambda_4$ &  $\lambda_5$ &  $\lambda_6$ & $\lambda_7$ 
    & $\lambda_8$ & $\lambda_9$ \\ 
\midrule 
 $C_9$ & 2 & 1.532 & 1.532 & 0.347 & 0.347 & $-1$ &
 $-1$ & $-1.879$ & $-1.879$ \\ 
$T_1$ & 2.223  &  1.568 & 1.247 & 0.288 & 0 & $-0.445$ 
& $-0.919$ & $-1.801$ & $-2.161$
 \\ 
$T_2$ & 2.573  &  1.453 & 1.441 & 0.566 & $-0.358$ & 
$-0.485$ & $-0.795$ & $-1.871$ & $-2.523$  \\ 
$T_3$ & 2.579 & 1.618 & 1.373 & 0 & 0 & $-0.451$ & $-0.618$ 
& $-2$ & $-2.501$   \\
$T_4$ & 2.503 & 1.813 & 1.264 & 0 &0 & $-0.470$ & $-0.576$ 
& $-2.191$ & $-2.342$ \\ 
$T_5$ & 2.414 & 1.508 & 1.247 & 0.679 & $-0.414$ & 
$-0.445$ & $-0.825$ & $-1.801$ & $-2.362$ \\ 
$T_6$ & 2.124 & 1.540 & 1.247 & 0.807 & $-0.337$ & 
$-0.445$ & $-1.101$ & $-1.801$ & $-2.032$ \\ 
\bottomrule 
\end{tabular} 
\caption{Eigenvalues of graphs $C_9$ and $T_i$ for $i\in \{1,2,\ldots ,6\}$.} \label{tab-2} 
 \end{table}

\begin{claim} \label{claim4.2}
Any graph of 
$\{T_i: 1\le i\le 6\}$ can not be an induced subgraph in $G$. 
\end{claim}

We denote by $T_0$ be the graph on $8$ vertices 
obtained from $C_7$ by hanging an edge. 
Unfortunately, we can not  prove  that 
$T_0$ is not an induced subgraph of $G$ by using similar calculations. 
This is slightly different from the proof of Theorem \ref{thmZS2022} in Section \ref{sec3}, and makes the forthcoming proof more complicated. 

\begin{proof}[Proof of Claim \ref{claim4.2}] 
Our proof needs some tedious calculations 
 similar with that in Claim \ref{claim4.1}.  
 Suppose on the contrary that $G$ contains $T_i$ 
 as an induced subgraph for some $i\in \{1,2,3,4,5,6\}$. 
 To obtain a contradiction, we shall show $t(G)>0$ by applying 
 Lemma \ref{lem21}. 
 If $T_1$ is an induced subgraph of $G$, then 
 Cauchy's interlacing theorem gives $\lambda_{n-9+i}(G)\le \lambda (T_1) 
 \le \lambda_i (G)$ for every $i\in \{1,2,\ldots ,9\}$. 
 In particular, we have $\lambda_2\ge 1.568$, 
 $\lambda_3\ge 1.247$ and $\lambda_4\ge 0.288$. Then 
 \begin{gather*}
 g(\lambda_2) \ge g(1.568) \ge 2.458\sqrt{m-4} + 3.855,\\
 g(\lambda_3) \ge g(1.247) \ge 1.555\sqrt{m-4} + 1.939,
 \end{gather*}
 and 
 \[ g(\lambda_4) \ge g(0.288) \ge 0.082 \sqrt{m-4} + 0.023.  \]
 In addition, the negative eigenvalues of $T_1$ imply that 
 $\lambda_{n-3}\le -0.445$, $\lambda_{n-2} \le -0.919$, 
 $\lambda_{n-1} \le -1.801$ and $\lambda_n \le -2.161$. 
We know from (\ref{eq-gamma}) that  $\lambda_1 \ge \gamma (m)> \sqrt{m-4}$, 
which yields 
 $\lambda_n^2\le 2m - (\sum_{i=1}^4 \lambda_i^2 + \lambda_{n-3}^2 + 
 \lambda_{n-2}^2 + \lambda_{n-1}^2) < 2m- (m-4+8.382) =m-4.382$. 
 Then $-\sqrt{m-4.382} < \lambda_n \le -2.161$ and 
 \[  g(\lambda_n) \ge \min \{g(-\sqrt{m-4.382}), g(-2.161)\}>0.15\sqrt{m-4}. \]
 Since $\lambda_{n-1}^2 + \lambda_n^2 \le 
 2m- (\sum_{i=1}^4 \lambda_i^2 + \lambda_{n-3}^2 + 
 \lambda_{n-2}^2) < 2m-(m-4+5.139)=m-1.139$ and 
 $\lambda_{n-1}^2 \le \lambda_n^2$, we get 
 $-\sqrt{(m-1.139)/2}< \lambda_{n-1} \le -1.801$ and 
\[  g(\lambda_{n-1}) \ge \min \{g(-\sqrt{(m-1.139)/2}), g(-1.801)\}\ge 
3.243\sqrt{m-4} - 5.841. \]
Similarly,  we have $-\sqrt{(m-0.294)/3} < \lambda_{n-2} \le -0.919$ and 
 \[    g(\lambda_{n-2}) \ge \min \{g(-\sqrt{(m-0.294)/3}), g(-0.919)\}\ge 
0.844\sqrt{m-4} - 0.776. \]  
 By Lemma \ref{lem21}, we have 
\begin{align*}  t(G) &> 
\frac{1}{6} (g(\lambda_2) + g(\lambda_3) + g(\lambda_4) + g(\lambda_n) + 
g(\lambda_{n-1}) + g(\lambda_{n-2})) - \frac{4}{3}\lambda_1(G) \\
&>\frac{1}{6}( 8.332\sqrt{m-4} -8\sqrt{m} - 0.8)>0, 
\end{align*}
which is a contradiction.

If $T_2$ is an induced subgraph of $G$, then 
Cauchy's interlacing theorem implies $\lambda_{n-9+i}(G)\le \lambda_i(T_2) \le \lambda_i(G)$ 
for every $i\in \{1,2,\ldots ,9\}$. 
Since $\lambda_2 \ge 1.453$, $\lambda_3\ge 1.441$ and $\lambda_4 \ge 
0.566$, we obtain 
\begin{gather*} 
g(\lambda_2) \ge g(1.453)\ge 2.111\sqrt{m-4} + 3.067,  \\
g(\lambda_3) \ge g(1.441)\ge 2.076 \sqrt{m-4} + 2.992,
\end{gather*}
and 
\[  g(\lambda_4) \ge g(0.566) \ge 0.320\sqrt{m-4} + 0.181. \]
On the other hand, the negative eigenvalues of $T_2$ can imply that 
$\lambda_{n-4}\le -0.358$, $\lambda_{n-3} \le -0.485$, 
$\lambda_{n-2}\le -0.795$, $\lambda_{n-1} \le -1.871$ 
and $\lambda_n \le -2.523$. 
Due to $\lambda_1> \sqrt{m-4}$, then we have  
$\lambda_n^2 \le 2m- (\sum_{i=1}^4 \lambda_i^2 + \lambda_{n-i}^2) < 2m- (m-4 +9.004) \le m-5.004$, which yields 
$-\sqrt{m-5.004} < \lambda_n \le -2.523 $. Consequently, we get 
\[  g(\lambda_n) \ge \min\{g(-\sqrt{m-5.004}), g(-2.523)\} > 
0.4\sqrt{m-4}.  \]
Moreover, since $\lambda_{n-1}^2 + \lambda_n^2 \le 
2m- (\sum_{i=1}^4 \lambda_i^2 + \lambda_{n-4}^2+\lambda_{n-3}^2+\lambda_{n-2}^2) < 2m-(m-4+5.503)=m-1.503$ and 
$\lambda_{n-1}^2 \le \lambda_n^2$, we get 
$-\sqrt{(m-1.503)/2} < \lambda_{n-1} \le -1.871$ and  
\[  g(\lambda_{n-1}) \ge 
\min\{ g(-\sqrt{(m-1.503)/2}), g(-1.871)\} \ge 3.5\sqrt{m-4}-6.549. \]
Similarly, we have $-\sqrt{(m-0.871)/3} < \lambda_{n-2} \le -0.795$ and 
\[  g(\lambda_{n-2}) \ge \min\{g(-\sqrt{(m-0.871)/3}), g(-0.795)\} 
\ge 0.632 \sqrt{m-4} - 0.502. \]
By Lemma \ref{lem21}, we have 
\begin{align*}  t(G) &> 
\frac{1}{6} (g(\lambda_2) + g(\lambda_3) + g(\lambda_4) + g(\lambda_n) + 
g(\lambda_{n-1}) + g(\lambda_{n-2})) - \frac{4}{3}\lambda_1(G) \\
&>\frac{1}{6}(9.039 \sqrt{m-4} -8\sqrt{m} - 0.811)>0, 
\end{align*}
which is a contradiction.

If $T_3$ is an induced subgraph of $G$, then 
$\lambda_2\ge 1.618$ and $\lambda_3\ge 1.373$. We get  
\[  g(\lambda_2)\ge g(1.618)\ge 2.617 \sqrt{m-4} + 4.235, \]
and 
\[ g(\lambda_3) \ge g(1.373) \ge 1.885\sqrt{m-4}+2.588.  \]
The negative eigenvalues of $T_3$ can give that 
$\lambda_{n-3}\le -0.451$, $\lambda_{n-2}\le -0.618$, 
$\lambda_{n-1}\le -2$ and $\lambda_{n}\le -2.501$. 
Since $\lambda_n^2 \le 2m- (\lambda_1^2+\lambda_2^2+\lambda_3^2 + 
\lambda_{n-3}^2+\lambda_{n-2}^2+\lambda_{n-1}^2) < 
2m- (m-4 + 9.088)=m-5.088$, 
then $-\sqrt{m-5.088} < \lambda_n \le -2.501$ and 
\[  g(\lambda_n)\ge \min\{ g(-\sqrt{m-5.088}), g(-2.501)\} 
> 0.5\sqrt{m-4}. \]
Since $\lambda_{n-1}^2+\lambda_n^2\le 
2m-(\lambda_1^2+\lambda_2^2+\lambda_3^2 + 
\lambda_{n-3}^2+\lambda_{n-2}^2)<2m-(m-4 + 5.088)=m-1.088$ 
and $\lambda_{n-1}^2\le \lambda_n^2$, we get 
$-\sqrt{(m-1.088)/2} < \lambda_{n-1} \le -2$ and 
\[ g(\lambda_{n-1}) \ge \min\{g(-\sqrt{(m-1.088)/2}), g(-2)\} 
\ge 4\sqrt{m-4} -8. \]
By Lemma \ref{lem21}, we have 
\begin{align*}  t(G) &> 
\frac{1}{6} (g(\lambda_2) + g(\lambda_3)  + g(\lambda_n) + 
g(\lambda_{n-1}) ) - \frac{4}{3}\lambda_1(G) \\
&>\frac{1}{6}( 9.002\sqrt{m-4} -8\sqrt{m} - 1.177)>0, 
\end{align*}
which is a contradiction.

If $T_4$ is an induced subgraph of $G$, 
then we get from Cauchy's interlacing theorem that 
$\lambda_2\ge 1.813$ and $\lambda_3\ge 1.264$. Then 
\[  g(\lambda_2) \ge g(1.813)\ge 3.286\sqrt{m-4} + 5.959, \]
and 
\[ g(\lambda_3) \ge g(1.264) \ge 1.597\sqrt{m-4} + 2.019. \]
Moreover, we have $\lambda_{n-3}\le -0.470$, 
$\lambda_{n-2}\le -0.576$, $\lambda_{n-1}\le -2.191$ 
and $\lambda_n \le -2.342$. 
Since $\lambda_{n}^2\le 2m- (\sum_{i=1}^3 \lambda_i+ \lambda_{n-i}) 
< 2m- (m-4 +10.237)=m-6.237$, we get $-\sqrt{m-6.237} 
< \lambda_n \le -2.342$ and 
\[  g(\lambda_n) \ge \min\{ g(-\sqrt{m-6.237}), g(-2.342)\} \ge \sqrt{m-4}. \]
Since $\lambda_{n-1}^2+\lambda_n^2 \le 
2m- (\sum_{i=1}^3 \lambda_i+ \lambda_{n-3} + \lambda_{n-2}) 
< 2m- (m-4+5.437)=m-1.437$ and $\lambda_{n-1}^2\le \lambda_n^2$, 
we get $-\sqrt{(m-1.437)/2} < \lambda_{n-1} \le -2.191$ and 
\[ g(\lambda_{n-1}) \ge \min\{g(-\sqrt{(m-1.437)/2}), g(-2.191)\} 
\ge 4.8\sqrt{m-4} -10.517.  \]
By Lemma \ref{lem21}, we obtain 
\begin{equation*}
\begin{aligned}  t(G) &> 
\frac{1}{6} (g(\lambda_2) + g(\lambda_3)  + g(\lambda_n) + 
g(\lambda_{n-1}) ) - \frac{4}{3}\lambda_1(G) \\
&>\frac{1}{6}( 10.683\sqrt{m-4} -8\sqrt{m} - 2.539)>0, 
\end{aligned}
\end{equation*}
which is a contradiction.

If $T_5$ is an induced subgraph of $G$, then 
Cauchy's interlacing theorem implies 
$\lambda_2\ge 1.508$, $\lambda_3\ge 1.247$ and $\lambda_4\ge 0.679$. 
Then 
\begin{gather*}
g(\lambda_2) \ge g(1.508) \ge 2.274\sqrt{m-4} + 3.429,\\
g(\lambda_3)\ge g(1.247) \ge 1.555\sqrt{m-4} + 1.939, 
\end{gather*}
and 
\[ g(\lambda_4)\ge g(0.679) \ge 0.461\sqrt{m-4} + 0.313. \]
The negative eigenvalues of $T_5$ imply that 
$\lambda_{n-4}\le -0.414$, $\lambda_{n-3}\le -0.445$, 
$\lambda_{n-2}\le -0.825$, $\lambda_{n-1}\le -1.801$ 
and $\lambda_n \le -2.362$. 
Since $\lambda_n^2\le 2m-(\sum_{i=1}^4 \lambda_i + \lambda_{n-i}) 
< 2m- (m-4+8.583)=m-4.583$, we get 
$-\sqrt{m-4.583} < \lambda_n \le -2.362$ and 
\[  g(\lambda_n) \ge \min\{g(-\sqrt{m-4.583}), g(-2.362)\} 
\ge 0.25\sqrt{m-4}. \]
Since $\lambda_{n-1}^2+\lambda_n^2\le 
2m- (\sum_{i=1}^4 \lambda_i^2 + \lambda_{n-4}^2 + \lambda_{n-3}^2 + 
\lambda_{n-2}^2)< 2m- (m-4+5.34)=m-1.34$ and 
$\lambda_{n-1}^2\le \lambda_n^2$, we have 
$-\sqrt{(m-1.34)/2} < \lambda_{n-1} \le -1.801$ and 
\[ g(\lambda_{n-1}) \ge \min\{g(-\sqrt{(m-1.34)/2}), g(-1.801)\} 
\ge 3.243\sqrt{m-4} -5.841.  \]
Similarly, we can get $-\sqrt{(m-0.659)/3} < \lambda_{n-2} \le -0.825$ 
and 
\[  g(\lambda_{n-2}) \ge \min\{ g(-\sqrt{(m-0.659)/3}), 
g(-0.825)\} \ge 0.68\sqrt{m-4} - 0.561.  \]
By Lemma \ref{lem21}, we obtain 
\begin{equation*}
\begin{aligned}  t(G) &> 
\frac{1}{6} (g(\lambda_2) + g(\lambda_3)  + g(\lambda_4) + g(\lambda_n) + 
g(\lambda_{n-1}) + g(\lambda_{n-2})) - \frac{4}{3}\lambda_1(G) \\
&>\frac{1}{6}( 8.463\sqrt{m-4} -8\sqrt{m} - 0.721)>0, 
\end{aligned}
\end{equation*}
which is a contradiction.

If $T_6$ is an induced subgraph of $G$, then Cauchy's interlacing theorem implies that 
$\lambda_2\ge 1.540$, $\lambda_3\ge 1.247$ and $\lambda_4\ge 0.807$. 
Then 
\begin{gather*}
g(\lambda_2)\ge g(1.540)\ge 2.371\sqrt{m-4} + 3.652,\\
g(\lambda_3) \ge g(1.247) \ge 1.555\sqrt{m-4} +1.939,
\end{gather*}
and 
\[  g(\lambda_4)\ge g(0.807) \ge 0.651\sqrt{m-4}+0.525. \]
The negative eigenvalues of $T_6$ yield that 
$\lambda_{n-4}\le -0.337$, $\lambda_{n-3}\le -0.445$, 
$\lambda_{n-2}\le -1.101$, $\lambda_{n-1} \le -1.801$ and 
$\lambda_n \le -2.032$. 
Since $\lambda_n^2 \le 2m- (\sum_{i=1}^4 \lambda_i^2 + \lambda_{n-i}^2) < 2m- (m-4 + 9.345)=m-5.345$, we get $-\sqrt{m-5.345} < \lambda_n \le -2.032$ and 
\[  g(\lambda_n) \ge \min\{g(-\sqrt{m-5.345}), g(-2.032)\} 
\ge 0.65\sqrt{m-4} . \]
Since $\lambda_{n-1}^2+\lambda_n^2 \le 
2m- (\sum_{i=1}^4 \lambda_i^2+\lambda_{n-4}^2+\lambda_{n-3}^2+\lambda_{n-2}^2) < 2m- (m-4+6.101) =m-2.101$ 
and $\lambda_{n-1}^2\le \lambda_n^2$, we get 
$-\sqrt{(m-2.101)/2} <\lambda_{n-1} \le -1.801$ and 
\[  g(\lambda_{n-1}) \ge \min\{ g(-\sqrt{(m-2.101)/2} ), g(-1.801)\} 
\ge 3.243\sqrt{m-4} -5.841.  \]
Similarly, we can get 
$-\sqrt{(m-0.889)/3} < \lambda_{n-2} \le -1.101$ and 
\[  g(\lambda_{n-2}) \ge \min\{g(-\sqrt{(m-0.889)/3}), g(-1.101)\} 
\ge 1.212\sqrt{m-4} - 1.334. \]
By Lemma \ref{lem21}, we obtain 
\begin{equation*}
\begin{aligned}  t(G) &> 
\frac{1}{6} (g(\lambda_2) + g(\lambda_3)  + g(\lambda_4) + g(\lambda_n) + 
g(\lambda_{n-1}) + g(\lambda_{n-2})) - \frac{4}{3}\lambda_1(G) \\
&>\frac{1}{6}( 9.682\sqrt{m-4} -8\sqrt{m} - 1.059)>0, 
\end{aligned}
\end{equation*}
which is a contradiction. 
\end{proof}

\begin{claim} \label{claim4.3}
$V(G)=S\cup N(S)$ and $d_S(v)\in \{1,2\}$ for each vertex $v\in N(S)$. 
\end{claim}

\begin{proof}[Proof of Claim \ref{claim4.3}]
For each  $v\in N(S)$, 
 without loss of generality, 
 we may assume that $v\in N(u_1)$. 
 If $d_S(v)\ge 3$, then  we can find either a $C_3$ or $C_5$ in $G$, a contradiction. 
 This implies that $d_S(v)\in \{1,2\}$ for every $v\in N(S)$.  
 Next we prove that $V(G)=S\cup N(S)$. 
 Otherwise, if there is a vertex $v' \in V(G) \setminus (S\cup N(S))$, then $v'$ has distance at least $2$ from $S$. 
Let $v'vu_1$ be a path of $G$ such that 
 $v'u_i \notin E(G)$ for every $i\in [7]$. 
Since $d_S(v)\in \{1,2\}$ and $G$ is both $C_3$-free and $C_5$-free,  
we know by symmetry that 
 either $N_S(v)=\{u_1\}$ or $N_S(v)=\{u_1,u_3\}$. 
 If $N_S(v)=\{u_1\}$, 
 then 
 $\{v',v\}\cup S$ induces a copy of $T_6$, 
 a contradiction. 
 If $N_S(v)=\{u_1,u_3\}$, then 
 $\{v',v\}\cup S$ induces a copy of $T_5$, which is a contradiction. 
 Thus, we conclude that $V(G)=S \cup N(S)$ and 
 $d_S(v)\in \{1,2\}$ for every $v\in N(S)$. 
 \end{proof}
 
 From Claim \ref{claim4.3}, we assume that $V(G)\setminus S=V_1\cup V_2$, 
 where $V_i=\{v\in N(S): d_S(v)=i\}$ for every $i=1,2$. 
 Since $T_1$ is not an induced subgraph of $G$, 
 we get $0\le |V_1|\le 1$. 
Since $m\ge 11$, we get $V_2\neq \varnothing$. 
We can fix a vertex $v\in V_2$ and 
assume that $N_S(v)=\{u_1,u_3\}$. 
For each $w\in V_2$, 
since $G$ contains no triangles and  no  
$T_3$  as  induced subgraphs, we know that 
$N_S(w)\neq \{u_3,u_5\}$ and $N_S(w)\neq \{u_6,u_1\}$. 
Similarly, since $G$ contains no pentagon and $T_4$ as induced subgraphs, 
we get $N_S(w)\neq \{u_4,u_6\}$ and $N_S(w)\neq \{u_5,u_7\}$. 
Therefore, it is possible that $N_S(w)=\{u_1,u_3\},\{u_2,u_4\}$ or 
$\{u_7,u_2\}$. Furthermore, 
if $N_S(w)=\{u_1,u_3\}$, then $wv \notin E(G)$, since $G$  contains no  triangles; 
if $N_S(w)=\{u_2,u_4\}$, then 
$wv \in E(G)$, since $G$ contains no induced copy of 
$T_2$. 
We denote $N_{i,j}=\{w\in V(G)\setminus S : N_S(w)=\{u_i,u_j\}\}$. Note that $G$ has no induced copy of $T_3$, 
 there are at least one empty set 
in $\{N_{2,4},N_{7,2}\}$. 

{\bf Case 1.}
 If $N_{2,4}=\varnothing$ and $N_{7,2}=\varnothing$, 
then $V_2 = N_{1,3}$ and $V(G)\setminus S = N_{1,3}\cup V_1$. 
If $|V_1|=0$, then  $m$ is odd and $G=S_3(K_{2,\frac{m-3}{2}})$, 
as desired. 
If $|V_1|=1$, then  $m$ is even and 
 $G$ is a graph obtained from $S_3(K_{2,\frac{m-4}{2}})$ 
by hanging an edge to one vertex. 
By setting $a=2$ and $b=\frac{m-4}{2}$ in 
 Lemma \ref{lem42}, we know that $\lambda (G) < \gamma (m)$.

{\bf Case 2.} Without loss of generality, 
we may assume that $N_{2,4} \neq \varnothing$, then 
$V(G)\setminus S= N_{1,3} \cup N_{2,4}\cup V_1$. 
Moreover, $N_{1,3}$ and $N_{2,4}$ induce a complete 
bipartite subgraph in $G$. 
We denote $A=N_{1,3} \cup \{u_2,u_4\}$ 
and $B=N_{2,4}\cup \{u_3,u_1\}$. 
Clearly, we have $|A|=a \ge 2$ and $|B|=b\ge 2$.  
If $|V_1|=0$, then 
$G$ is isomorphic to the subdivision of  $K_{a,b}$ by 
replacing the edge $u_1u_4$ of $K_{a,b}$ with a path of length $4$, 
and $m= ab +3$. Note that $\lambda (S_3(K_{a,b}))$ is the largest root of 
\begin{align*}
Q(x) &:= x^7-(ab+3)x^5 + (5ab -2a -2b +2)x^3 \\ 
& \quad + (-5ab +4a +4b -3)x  -2ab +2a +2b -2. 
\end{align*} 
Recall in (\ref{Lx}) that $\gamma (m)$ denotes the largest root of 
$ L(x)$.  
We can easily verify that $L(x)\le Q(x)$ for every $x\ge 1$, so we get $L(\lambda (G)) \le Q(\lambda (G))=0$, 
which implies $\lambda (G) \le \gamma(m)$, equality holds if and only if 
$a=2$ or $b=2$, and thus $m$ is odd and $G=S_3(K_{2,\frac{m-3}{2}})$.
If $|V_1|=1$, then $m=ab+4$ and 
$G$ is obtained from $S_3(K_{a,b})$ by hanging an edge to 
one vertex. 
By Lemma \ref{lem42} again, 
we get $\lambda (G)< \gamma (m)$. 
This completes the proof. 
\end{proof}

\section{Concluding remarks} 

\label{sec5}

 We remark that the method stated in Sections \ref{sec3} and \ref{sec4} 
 can further allow us to determine 
the largest spectral radius of non-bipartite graphs with no copy of $C_3,C_5$ and $C_7$, and so far as to $C_9$ by more careful computations. 
From this evidence, 
we propose the following  conjecture for interested readers. 
Let $S_{2k-1}(K_{s,t})$ denote the 
graph obtained from the complete bipartite graph 
$K_{s,t}$ by replacing an edge with a  path $P_{2k+1}$ 
on $2k+1$ vertices, 
that is, introducing $2k-1$ new vertices on an edge. 
Clearly,  the odd girth of $S_{2k-1}(K_{s,t})$ is $2k+3$. 

\begin{conjecture} 
Let $G$ be a graph with $m$ edges.  
If $G$ does not contain any member of $\{C_3,C_5, \ldots ,C_{2k+1}\}$
 and $G$ is non-bipartite, 
then 
\[  \lambda (G) \le \lambda (S_{2k-1}(K_{2,\frac{m-2k+1}{2}}) ),  \]
equality holds if and only if $m$ is odd and $G=S_{2k-1}(K_{2, \frac{m-2k+1}{2} })$.  
\end{conjecture}

Let $B_{k}$ be the book graph, i.e., 
the graph obtained from $k$ triangles by sharing a common edge. 
In particular, we have $B_1=K_3$ and $B_2=K_4^{-}$, the $4$-vertex complete graph minus an edge. 
In 2021, 
Zhai, Lin and Shu \cite[Conjecture 5.2]{ZLS2021} made the following conjecture: 
Let $m$ be large enough and $G$ be a $B_k$-free graph with $m$ edges. 
Then  
\[   \lambda (G)\le \sqrt{m},  \]  
equality holds if and only if $G$ is a complete bipartite graph.

Soon after, Nikiforov \cite{Niki2021} confirmed Zhai--Lin--Shu's conjecture 
 by showing the following stronger theorem. 
Let $bk(G)$ denote the booksize of $G$, that is, the maximum 
number of triangles with a common edge in $G$. 
Nikiforov \cite{Niki2021} proved that 
if $G$ is a graph with $m$ edges and $\lambda (G)\ge \sqrt{m}$, 
then 
\[  bk(G)> \frac{1}{12}\sqrt[4]{m}, \] 
unless  $G$ is a complete bipartite graph (with possibly some isolated vertices). 
Since $B_2$ contains both $C_3$ and $C_4$ as a subgraph, 
 the result of Nikiforov generalized the Nosal Theorem \ref{thmnosal}.

 We conclude this paper with the following problem and 
  conjecture that the lower bound  $bk(G)\ge c\sqrt{m}$ 
 is true for some constant $c>0$.

\begin{conjecture} 
If $G$ is a graph with $m$ edges and $\lambda (G)\ge \sqrt{m}$, 
then 
\[  bk(G)\ge c \sqrt{m} \] 
for some constant $c>0$, unless  $G$ is a complete bipartite graph. 
\end{conjecture}

\subsection*{Acknowledgements} 
This work was supported by  NSFC (Grant No. 11931002).  
We would like to express sincere thanks to 
Huiqiu Lin, Bo Ning and Mingqing Zhai 
for kind discussions, which considerably improves the presentation of the manuscript.

\end{document}